\documentclass[11pt,a4paper,reqno]{amsart}
\usepackage{amsaddr}
\usepackage{amsmath,amsfonts,amssymb,amsthm,mathrsfs}
\usepackage[utf8]{inputenc}
\usepackage[T1]{fontenc}
\usepackage{lmodern}
\usepackage{latexsym}
\usepackage{cancel}
\usepackage{microtype}
\usepackage{caption}
\usepackage{MnSymbol}
\usepackage{xcolor}
\usepackage{enumerate}
\usepackage[normalem]{ulem}
\usepackage{bbm}
\usepackage{geometry}
\usepackage{marginnote}
\usepackage{mathrsfs}

\usepackage[colorlinks=true,linkcolor=blue,citecolor=magenta]{hyperref}
\usepackage{graphics,tikz}

\newcommand{\FF}{{\mathcal  F}}

\def\XXint#1#2#3{{\setbox0=\hbox{$#1{#2#3}{\int}$ }
\vcenter{\hbox{$#2#3$ }}\kern-.6\wd0}}

\newtheorem{theorem}{\bf Theorem}[section]
\newtheorem{proposition}[theorem]{\bf Proposition}
\newtheorem{lemma}[theorem]{\bf Lemma}
\newtheorem{corollary}[theorem]{\bf Corollary}

\theoremstyle{definition}
\newtheorem{definition}[theorem]{Definition}
\newtheorem{example}[theorem]{\bf Example}

\newtheorem{remark}[theorem]{Remark}
\newtheorem{hypothesis}[theorem]{Hypothesis}

\numberwithin{equation}{section}

\begin{document}

\title[B\^ocher type theorem]{B\^ocher  type theorem for    elliptic equations with    
drift perturbed L\'evy operator}

\maketitle
\begin{center}

  \normalsize
  TOMASZ KLIMSIAK\footnote{e-mail: {\tt tomas@mat.umk.pl}}\textsuperscript{1,2}
   \par \bigskip

  \textsuperscript{1} {\small Faculty of
Mathematics and Computer Science, Nicolaus Copernicus University,\\
Chopina 12/18, 87-100 Toru\'n, Poland }\par \medskip

  \textsuperscript{2} {\small Institute of Mathematics, Polish Academy Of Sciences,\\
 \'{S}niadeckich 8,   00-656 Warsaw, Poland} \par
\end{center}

\begin{abstract}
A classical B\^ocher's theorem  asserts that any positive   
 harmonic function (with respect to the Laplacian)  in the punctured unit ball  
can be expressed, up to the multiplication constant,  as the sum    of 
the Newtonian kernel   and  a positive  function
that is harmonic in the whole unit ball. 
This theorem expresses   one of the fundamental results 
in the theory of isolated singularities and it  
 can be viewed as a statement on the asymptotic behavior 
of  positive harmonic functions near their isolated singularities.
In the paper we generalize this results to drift perturbed L\'evy operators.
We propose a new approach based on the probabilistic potential theory.
It  applies  to
L\'evy operators  for which the resolvent of its perturbation is strongly Feller.
In particular our result  encompasses drift  perturbed  fractional Laplacians
with any stability index  bounded between zero  and two -  the method therefore applies
to subcritical and supercritical cases.

\end{abstract}
\maketitle


\section{Introduction}
\label{sec1}

\subsection{B\^ocher's theorem}

A classical B\^ocher's theorem (see \cite{Bocher})  asserts that any    positive
   function $u$ on $\mathbb R^d$ ($d\ge 3$) that is harmonic  in the  unit punctured ball:
 \begin{equation}
 \label{eq1.1aaa}
 -\Delta u(x)=0,\quad x\in B_1\setminus \{0\}=: B^*_1
 \end{equation}
admits the following decomposition
 \begin{equation}
 \label{eq1.1bbb}
u(x)=h(x)+a|x|^{2-d},\quad x\in B(0,1),
\end{equation}
where  $h:B_1\to \mathbb R^+$ is  harmonic      and $a\ge 0$ is a constant.
The  equality  may be viewed as    a result  on  the  local behavior of  positive solutions to \eqref{eq1.1aaa} 
around isolated singularity.
If $a=0$ 
singularity is removable and if $a>0$, we have $u(x)\sim |x|^{2-d}$ when $|x|\to 0$.
Further generalizations of B\^ocher's theorem  have been made by   Hartman and Wintner (see \cite{HW})
and  Gilbarg and Serrin (see \cite{GS})
for classical solutions to   
\[
a_{ij} u_{x_ix_j}+b_i u_{x_i}+c u=0\quad \text{ in }B^*_1,
\]
assuming  H\"older continuity of  coefficients (non-divergence form operators), and by Serrin and  Weinberger (see \cite{SW})
for positive continuous weak solutions ($u\in H^1_{loc}(B^*_1)$) to  
\[
(a_{ij} u_{x_i})_{x_j}=0\quad \text{ in }B^*_1,
\]
with merely measurable bounded and uniformly elliptic matrix $(a_{ij})$ (divergence form operators, see also \cite{Serrin1,Serrin2}). 
The first result on   B\^ocher's type theorem for  distributional  supersolutions
is due to  Brezis and Lions  \cite{BL}: assume that $u\in L^1_{loc}(B^*_1)$ is positive, $\Delta u\in L^1_{loc}(B^*_1)$
in the sense of distributions and   
\[
\Delta u\le cu+f\quad\text{a.e. in } B_1, 
\]
for some $c\ge 0$ and $f\in L^1_{loc}(B_1)$. Then $u\in L^1_{loc}(B_1)$, and 
\[
-\Delta u= \varphi+a\delta_{0}\quad\text{in } \mathscr D'(B_1)
\]
for some $\varphi\in L^1_{loc}(B_1)$ and $a\ge 0$ (in the paper $d\ge 2$, 
and note that one easily shows that the result does not hold for $d=1$). 
Noticing that $-\Delta |x|^{2-d}=c\delta_0$ in $\mathscr D'(\mathbb R^d)$ it is readily seen that
the result by Brezis and Lions is a generalization of the one by B\^ocher.
These kinds of results became  fundamental for the theory of isolated singularities
that has been extensively studied since 1950'. We restrict ourself to mentioning 
Gilbarg and Serrin \cite{GS}, Serrin \cite{Serrin1,Serrin2}, DeGiorgi and   Stampacchia \cite{DS},  
Serrin and Weinberger \cite{SW}, Lions \cite{Lions}, Brezis, Veron \cite{BV}, 
Veron \cite{Veron,Veron1}, Aviles \cite{Aviles}, Bers \cite{Bers}, Gidas and Spruck \cite{GiSp}, 
Caffarelli, Gidas and Spruck \cite{CGS}, Congming Li \cite{Congming}, Chen and Lin 
\cite{CL} (here $K$ is not a singleton), Labutin \cite{Labutin},   Barras and Pierre \cite{BP}, Vazquez and Veron
\cite{VV}, Caffarelli, Yanyan Li, Nirenberg \cite{CLN}, Xiong \cite{Xiong}.
For recent results on B\^ocher's type theorems, see e.g.  
\cite{LW,Taylor,Wan} and the references therein.

\subsection{B\^ocher's theorem for the fractional Laplacian}
In the last decade the problem of local behavior of  solutions to semilinear equations with isolated singularities
for the fractional Laplacian  has attracted 
a considerable interest  with emphasis on the fractional Lane-Emden equations
and fractional Emden-Fowler equations. We mention here the paper by Caffarelli, Yin  and  Sire \cite{CJS}, Gonzalez, Mazzeo and Sire \cite{GMS}, DelaTorre, del Pino, Gonz\'alez and Wei   \cite{DDGW},  Chen and Quaas \cite{CQ}, DelaTorre and Gonz\'alez \cite{DG},
Jin,  de Queiroz,  Sire and Xiong  \cite{JDSX}, Ao, DelaTorre, Gonz\'alez and   Wei \cite{ADGW}, 
Yang and Zou \cite{YZ}, Li and Bao \cite{LB}, Chen and Veron \cite{CH}. In most of the papers on the subject 
and all the papers mentioned here except \cite{CQ} the Caffarelli-Silvestre extension method has been
applied. This method although very fruitful is limited mostly to the fractional Laplacian. 
In \cite{CQ} the authors prove  and apply  B\^ocher's  type theorem for the fractional Laplacian.

To the best of our knowledge, there are only two papers  devoted to  B\^ocher's type theorems for non-local equations.
In \cite{LWX} Congming Li, Zhigang Wu and Hao Xu  studied  behavior  of positive $s$-harmonic  functions in $B^*_1$ (i.e. harmonic with respect to
the fractional Laplacian $\Delta^s:=-(-\Delta)^s$ with $s\in (0,1)$)
while in \cite{LLWX} Congming Li, Chenkai Liu, Zhigang Wu and Hao Xu were concerned  with positive distributional solutions to the following inequality:
\begin{equation}
\label{eq1.0}
-\Delta^s u+b\cdot\nabla u+ cu \ge 0\quad\text{in}\quad B^*_1,
\end{equation}
with $s\in (1/2,1)$,  $b\in C^1(B_1)$ and $c\in L^\infty(B_1)$.
They proved that if $u\in L^1_{w_s}(\mathbb R^d)$, with $w_s(x):= 1/(1+|x|^{d+2s})$, is a distributional solution to \eqref{eq1.0}, then 
\[
-\Delta^su + b\cdot \nabla u+cu=\mu+a\delta_{0}\quad \text{in } \mathscr D'(B_1)
\]
for some positive Radon measure $\mu$ on $B_1$, with $\mu(\{0\})=0$,
and $a\ge 0$. The critical (i.e. $s=1/2$) and supercritical (i.e. $s\in (0,1/2)$) cases have been left as 
open problems. Their method, that is a modification of Lions' proof,
falls short of providing the proof of B\^ocher's theorem 
in the full range of exponent $s\in (0,1)$ since  for $s$ not exceeding $1/2$ 
the Green function of $\Delta^s$ does not belong to the space $W^{1,1}_{loc}(\mathbb R^d)$,
which was one of  the crucial points of the method.

\subsection{A new method for a wide class of local,  non-local and mixed operators }
In the present paper we propose a completely new approach to the proofs
of B\^ocher's type theorems  based on the probabilistic potential theory. The approach allows us to 
deal with a wide class of integro-differential operators including, as  particular examples, 
\begin{itemize}
\item the fractional Laplacians $\Delta^s$ with any $s\in (0,1)$ (see Example \ref{ex11});
\item L\'evy operators  with singular  (with respect to the Lebesgue measure)
L\'evy measure $\nu$ (see \eqref{L}), e.g. cylindrical L\'evy operators (see Example \ref{ex12});
\item Kolmogorov type operators with degenerate diffusion part (local and mixed, see Example \ref{ex13});
\item mixed local and non-local operators, e.g. $\Delta+\Delta^s$ (see Example \ref{ex14}).
\end{itemize} 
In particular, in some instances (e.g. degenerate diffusion matrix) our results are new even for local operators.
We believe that this result may pave the way for studying the problem of 
isolated singularities for a much broader class of non-local 
operators than that determined by the applicability of the Caffarelli-Silvestre extension procedure.
 
Let $D\subset \mathbb R^d$  be a bounded open set and $d\ge 1$
(note that we do not impose any restrictions on the dimension
but  this is not contradictory to the result by Brezis and Lions; the dimension constraint is in a sense 
forced by a specific operator).
We shall study   non-negative  solutions 
to the  inequality
\begin{equation}
\label{eq1.1}
-Au+b\cdot\nabla u+\lambda \ge 0\quad\text{in}\quad \mathscr D'(D\setminus K),
\end{equation}
where  $A$ is a  L\'evy operator,   $b:D\to \mathbb R^d$ is a bounded and Lipschitz continuous function, $\lambda$ is a Radon measure on $D$ and 
 $K$ is a {\em polar} compact set.  Note that    semilinear equations with isolated singularities  
 for  polar sets $K$  not being singletons    have been 
 considered by Brezis and Nirenberg  \cite{BN}, Jin,  de Queiroz,  Sire and Xiong  \cite{JDSX}, 
 and Chen and Lin \cite{CL}.  
The class of L\'evy operators consists of Fourier multiplier operators $A$:  
\[
\widehat{Au}(\xi):=-\psi(\xi)\hat u(\xi),\quad \xi\in\mathbb R^d,\, u\in C_c^2(\mathbb R^d),
\]
with a  negative definite (in the sense of
Schoenberg, see \cite[Definition 2.14]{BSW}) function $\psi:\mathbb R^d\to\mathbb C$.
Equivalently,  by the L\'evy-Khintchine formula,  the said class  consists of linear operators $A$ of the  form 
\begin{equation}
\label{L}
Au(x)=cu-l\cdot\nabla u+\frac12\text{Tr}(QD^2 u(x))+\int_{\mathbb R^d}(u(x+y)-u(x)-y\nabla u(x)\mathbf1_{B(0,1)})\,\nu(dy),
\end{equation}
where $c\in\mathbb R$, $l\in\mathbb R^d$, $Q$ is a  non-negative definite symmetric  matrix and $\nu$ is a  L\'evy measure - positive Borel  measure on $\mathbb R^d$
such that  $\nu(\{0\})=0$ and $\int_{\mathbb R^d}1\wedge |x|^2\,\nu(dx)<\infty$.
Furthermore,  if we assume additionally that $l=0$ and $\nu$ is  symmetric:  $\nu(dx)=\nu(-dx)$, then 
\[
Au(x)= cu+\frac12\text{Tr}(Q\nabla u(x)\cdot\nabla u(x))+\lim_{\varepsilon\searrow 0} \int_{\mathbb R^d\setminus B_\varepsilon}(u(x+y)-u(x))\,\nu(dy).
\]
The  only restriction that we  make on $A$ is strong Feller property of the resolvent  
generated by
$A-b\cdot\nabla-\kappa I$  for some $\kappa\ge 0$. 
Recall that $K$ is polar if $C(K)=0$, where
\[
C(K):=\sup\{\mu(K): \mu \text{ is a positive Borel measure supported in } K \text{ and } G\mu\le 1\},
\]
and $G$ is the Green function of $A-b\cdot\nabla-\kappa I$.
It is  known that  $C(\{x\})=0$ if and only if  $G(x,x)=\infty$. 
All the above assumptions are easily verified for $A=\Delta^{s}$ with any $s\in (0,1]$
and $K=\{x\}$ for any $x\in \mathbb R^d$ (see Example \ref{ex11}).
The method therefore applies
to subcritical and supercritical cases.  

\subsection{The main results of the paper}
Our main results consist of three theorems.  
Note at this point that local integrability of $u$ on $D\setminus K$ in general is not enough to
guarantee that $-Au+b\cdot\nabla u+\lambda$ be a distribution on $D\setminus K$.
This is due to the presence of the nonlocal part of $A$. Therefore, if $\nu\equiv 0$, we merely assume that $u\in L^1_{loc}(D\setminus K)$,
but for the general case, we introduce
in Section \ref{sec.dis}  the space $\mathscr T_\nu(D)$,
depending on  $\nu,D$,   that ensures that for any $u\in\mathscr T_\nu(D)\cap L^1_{loc}(\mathbb R^d)$,
the functional $-Au+b\nabla u+\lambda$,  acting by the following rule
\[
[-Au+b\nabla u+\lambda](\eta):= \int_{\mathbb R^d}u \,(-A+b\cdot\nabla)^*\eta
+\int_{D\setminus K}\eta\,d\lambda,\quad  \eta\in  C_c^\infty(D\setminus K),
\]
defines a distribution on $D\setminus K$, in other words, $-Au+b\nabla u+\lambda\in\mathscr D'(D\setminus K)$.
Let $\hat  \mu_0$ be  a Radon measure on $D\setminus K$ that represents positive distribution of \eqref{eq1.1}
(Riesz's theorem).
Let $\mu_0$ be the extension of $\hat\mu_0$ to $D$ by letting zero on $K$: $\mu_0(A):=\hat\mu_0(A\cap (D\setminus K))$,
$A\in\mathcal B(D)$.

The first main result  (see Theorem \ref{th.main1})  stands that  for any positive  solution $u$ to \eqref{eq1.1}
measure $\mu_0$ is Radon on $D$, $u\in L^1_{loc}(D)$,   and $u$   solves
\begin{equation}
\label{eq1.1a}
-Au+b\cdot\nabla u+\lambda = \mu_0+\sigma_K\quad\text{in}\quad \mathscr D'(D),
\end{equation}
where  $\sigma_K$ is a   positive Radon measures on $D$ supported in $K$. 
The second theorem  (see Theorem \ref{th.main2}) asserts that   if 
\[
D\ni x\longmapsto \int_DG_D(x,y)\,d\lambda^+(y)\in L^1(D),
\]
then any positive  solution $u$ to \eqref{eq1.1}  admits  finely continuous $\ell^d$-version $\hat u$ in $D\setminus N$,
where $N$ is a polar set described precisely in the assertion,  
and  for any open $V\subset\subset D$  there exists a positive  
function $h_V$ that is harmonic in  $V$  (with respect to $-A+b\cdot\nabla$)  such that 
\[
 \hat u(x)=h_V(x)+ \int_V G_V(x,y)\,\mu_0(dy)+\int_V G_V(x,y)\,\sigma_K(dy)-
 \int_V G_V(x,y)\,\lambda(dy),\quad x\in V,
\]
where $G_V$ is the Green function   for $V$ and $-A+b\cdot\nabla $.
In fact, denoting by $P_V$  the Poisson kernel for $V$ and $-A+b\cdot\nabla$, we have
\[
h_V(x)=\int_{V^c}\hat u(y)P_V(x,dy),\quad x\in V.
\]
In particular, if $\lambda=\mu_0=0$, i.e.  
$-Au+b\cdot\nabla u=0$ in $\mathscr D'(D\setminus K)$,
and $K=\{x_0\}$, then (cf. \eqref{eq1.1bbb})
\[
 \hat u(x)=h_V(x)+G_V(x,x_0),\quad x\in V.
\]
The third main result is a direct consequence  of the method we  use here.
We obtain the following maximum principle (see Theorem \ref{th.main2}):  let $u$ be a positive  solution  to \eqref{eq1.1}
and $\lambda^+\equiv 0$, then for any open $V\subset\subset D$,
\[
\hat u(x)\ge \inf_{y\in D\setminus V} \hat u(y)w_V(x),\quad x\in V,
\]
where
\[
w_V(x):=P_V(x,D\setminus V)=1-\int_{D^c}\int_V G_V(x,y)\nu(y-dz)\,dy,\quad x\in V
 \]
(the last equality is a consequence of the  Ikeda-Watanabe formula). In particular if either strong maximum
principle holds for $A-b\nabla$ (e.g. when $G_D(x,y)>0,\, x,y\in D$), or more directly,    $\nu(y-D\setminus \overline{V})>0$ a.e. in $y\in V$, then
\[
\inf_{x\in V}w_V(x)>0.
\]
Observe that in case $A$ is local, i.e. $\nu\equiv 0$, we have  $w_V\equiv 1$. Thus, interestingly,  in this case 
we get estimates independent of $b$ (cf. \cite[Theorem 1.4]{LLWX}).


\section{Preliminaries, L\'evy  and  Feller generators}

Let $d\ge 1$ and $\ell^d$ denote the Lebesgue measure on $\mathbb R^d$.
By $\mathcal B(\mathbb R^d)$ we denote the Borel $\sigma$-algebra of subsets of $\mathbb R^d$.
Let $\mathscr B^*(\mathbb R^d)$ denote the set of all universally measurable functions on $\mathbb R^d$.
For any $V\in \mathcal B(\mathbb R^d)$, we let $\mathscr B(V)$ denote the family
of all numerical functions $f$ on $V$, i.e. $f:V\to \overline{\mathbb R}:=\mathbb R\cup\{+\infty\}\cup\{-\infty\}$,
that are Borel measurable. 
For given $r>0$ and $x\in\mathbb R^d$ we let $B_r(x):=\{y\in\mathbb R^d: |x-y|<r\}$,
where $|x|$ denote the Euclidean  norm. We also use the shorthand $B_r:=B_r(0)$.
For open sets $V\subset\mathbb R^d$, we denote by  $C_c^\infty(V)$
the space of all smooth functions with the support included in $V$. 
We  write   $\mathscr D(V):= C_c^\infty(V)$.
Let $C_0(\mathbb R^d)$ denote the set of all continuous functions on $\mathbb R^d$
 vanishing at infinity. We equip $C_0(\mathbb R^d)$ with the  supremum norm:  
 $\|u\|_\infty:=\sup_{x\in\mathbb R^d}|u(x)|,\, u\in C_0(\mathbb R^d)$. 
 For a given class of numerical functions  $\mathscr A$ we denote by $p\mathscr A$ the subclass
 of $\mathscr A$ consisting of non-negative  functions (i.e. $[0,+\infty]$ valued).

Recall that a semigroup of bounded linear operators $(T_t)$ on $C_0(\mathbb R^d)$, i.e. $T_0=I$ and $T_{s+t}=T_s\circ T_t,\,s,t\ge 0$,
is called a {\em Feller semigroup}
provided that:
\begin{enumerate}[1)]
\item (Markovianity)   $f\in C_0(\mathbb R^d)$ and $0\le f\le 1$ implies  $0\le T_tf\le 1$ for any $t>0$;
\item (strong continuity)  for any $f\in C_0(\mathbb R^d)$,
\[
\|T_tf-f\|_\infty\to 0,\quad t\to \infty.
\]
\end{enumerate}
A linear operator $B$ on $C_0(\mathbb R^d)$ with a dense domain $\mathscr D(B)\subset C_0(\mathbb R^d)$
is called a {\em Feller generator} if it generates a Feller semigroup on $C_0(\mathbb R^d)$.
We say that a set $\mathcal C \subset C_0(\mathbb R^d)$ is a {\em core} of a Feller generator $(B,\mathscr D(B))$
if  $(B,\mathcal C)$ is closable in $C_0(\mathbb R^d)$ and its closure equals $(B,\mathscr D(B))$.
It is well known (see e.g. \cite[Corollary 2.10]{BSW}) that the closure of $(A,C_c^2(\mathbb R^d))$,  where $A$ is a L\'evy operator,
yields  a Feller generator. 

In what follows $A$ denote a L\'evy operator of the form \eqref{L}, with  $c\equiv 0$. 
Furthermore, by $A^*$ we denote the L\'evy operator that admits representation \eqref{L} with 
$(l,Q,\nu)$ replaced by  $(l^*,Q,\nu^*)$, where $l^*:= -l$ and $\nu^*(dx):=\nu(-dx)$.
We let $(A,\mathscr D(A))$, $(A^*,\mathscr D(A^*))$ denote the closures of
$(A,C_c^2(\mathbb R^d))$, $(A^*,C_c^2(\mathbb R^d))$, respectively,  and  
 by  $(S_t)$, $(S^*_t)$  their respective  semigroups on $C_0(\mathbb R^d)$.

\begin{remark}
\label{rem.adj}
By \cite[Proposition II.1]{Bertoin}   for any $\eta_1,\eta_2\in \mathscr D(A)\cap \mathscr D(A^*)$,
\begin{equation}
\label{eq.dual}
\int_{\mathbb R^d}A\eta_1\eta_2=\int_{\mathbb R^d}\eta_1A^*\eta_2
\end{equation}
provided that the integrals are well defined and finite. As a result
\[
\int_{\mathbb R^d}S_t\eta_1\eta_2=\int_{\mathbb R^d}\eta_1S^*_t\eta_2,\quad t\ge 0
\]
for any $\eta_1,\eta_2\in C_0(\mathbb R^d)$, provided that the integrals are well defined and finite.
 \end{remark}

\subsection{Drift  perturbed L\'evy operators}

We  let  $I:\mathscr B^*(\mathbb R^d)\to \mathscr B^*(\mathbb R^d)$ denote the identity mapping and for given $\kappa\ge 0$,
\[
L_\kappa:= A-\kappa I-b\cdot\nabla.
\]
By \cite[Theorem 137]{Situ}, \cite[Corollary 2.5,Theorem 3.1]{Kurtz}  the martingale problem for 
the operator $(L_\kappa,C_c^2(\mathbb R^d))$ is well posed. As a result, by \cite[Theorem 1.1]{kuhn},
for any $\kappa\ge 0$ there exists a unique Feller generator
$(L_\kappa,\mathscr D(L_\kappa))$  that is the extension of 
$(L_\kappa,C_c^2(\mathbb R^d))$ in $C_0(\mathbb R^d)$.
 For 
\[
\kappa \ge  \kappa_0,\quad \kappa_0:=\|\text{div} b\|_\infty,
\]
we let 
\begin{equation}
\label{eq2.dugen}
L^*_\kappa:= A^*+(\text{div}b-\kappa)I+b\cdot\nabla.
\end{equation}
By the same arguments as for $L_\kappa$ if  $\text{div}(b)\in C_b(\mathbb R^d)$, 
there exists a unique Feller generator
$(L^*_\kappa,\mathscr D(L_\kappa))$  that is the extension of 
$(L^*_\kappa,C_c^2(\mathbb R^d))$ in $C_0(\mathbb R^d)$.
We let $(T_t^\kappa)$ denote the Feller semigroup on $C_0(\mathbb R^d)$
generated by $(L_\kappa,\mathscr D(L_\kappa))$.
Let  $(T^{[\varepsilon],\kappa}_t)$  denote the Feller semigroup on $C_0(\mathbb R^d)$
generated by $(L^{[\varepsilon]}_\kappa,\mathscr D(L^{[\varepsilon]}_\kappa))$  which  is defined as
$(L_\kappa,\mathscr D(L_\kappa))$ but   with $b$
replaced by $b_\varepsilon$ that   is a standard mollification of $b$. 
Furthermore, let 
$(T^{[\varepsilon],\kappa,*}_t)$  denote the Feller semigroup on $C_0(\mathbb R^d)$
generated by $(L^{[\varepsilon],*}_\kappa,\mathscr D(L^{[\varepsilon],*}_\kappa))$  where  
$L^{[\varepsilon],*}_\kappa$  is given by \eqref{eq2.dugen}     with $b$
replaced by $b_\varepsilon$. 
By \cite[Theorem 19.25]{Kallenberg}
$T^{[\varepsilon],\kappa}_t f\to T_t^\kappa f$ for any $f\in C_0(\mathbb R^d)$.
By Remark \ref{rem.adj} for any positive $f\in C_0(\mathbb R^d)$ and $\kappa\ge \kappa_0$
\[
\int_{\mathbb R^d} T_t^\kappa f=\lim_{\varepsilon\to 0^+}\int_{\mathbb R^d} T^{[\varepsilon],\kappa}_tf=\lim_{\varepsilon\to 0^+}\int_{\mathbb R^d} fT^{[\varepsilon],\kappa,*}_t1\le 
\int_{\mathbb R^d} f.
\]
Now, by using Riesz-Thorin interpolation theorem, we conclude that for any $\kappa\ge \kappa_0$ 
$(T_t^\kappa)$  may be extended to
$L^p$-submarkovian semigroup on $L^p$ ($p>1$): 
there exists strongly continuous semigroup $(T_t^{\kappa,(p)})$
on $L^p(\mathbb R^d)$ such that
\[
\|T_t^{\kappa,(p)}f\|_{L^p}\le \|f\|_{L^p},
\]
and a.e.
\[
T_t^{\kappa,(p)}f=T^\kappa_tf,\quad f\in C_c(\mathbb R^d).
\]
We let $(L^{(p)}_\kappa,\mathscr D(L^{(p)}_\kappa))$ denote the operator generated by $(T_t^{\kappa,(p)})$.
One easily checks that $L^{(p)}_\kappa$ is the closure of $(L_\kappa,\mathscr D(L_\kappa)\cap L^p(\mathbb R^d))$.
By \cite[Corollary 1.10.6]{Pazy} for any $p>1$ the dual semigroup  $(T_t^{\kappa,(p'),*})$
is a $C_0$-semigroup of contractions that generates operator $(L^{(p),*}_\kappa, \mathscr D(L^{(p),*}_\kappa))$ that 
is the dual of $(L^{(p)}_\kappa,\mathscr D(L^{(p)}_\kappa))$ (see, \cite[(1.10.1)]{Pazy}).  Summing up,
for  $\kappa\ge \kappa_0$ and $p>1$,
\begin{equation}
\label{eq.lp}
\|T_t^{\kappa,(p)}f\|_{L^p}\le \|f\|_{L^p},\quad \|T_t^{\kappa,(p),*}f\|_{L^p}\le \|f\|_{L^p},\quad f\in L^p(\mathbb R^d).
\end{equation}
It is an elementary check that $L^{(p),*}_\kappa$ is the closure of $(L_\kappa^*,\mathscr D(A^*+b\cdot\nabla)\cap L^p(\mathbb R^d))$
in $L^p(\mathbb R^d)$.

\subsection{Markov processes and related semigroups}
\label{sec.fprs}

Let $\partial$ be a one-point compactification of $\mathbb R^d$
and $\Omega$ denote the set of all functions $\omega:[0,\infty)\to \mathbb R^d\cup\{\partial\}=:\mathbb R^d_\partial$
such that: 
\begin{enumerate}
\item[a)] $\omega$ is right-continuous on $[0,\infty)$ and possesses left limits on $(0,\infty)$ (c\`adl\`ags), 
\item[b)]
if $\omega(t)=\partial$ for some $t\in [0,\infty)$, then $\omega(s)=\partial$ for any $s\ge t$.
\end{enumerate}
We equip $\Omega$ with the Skorochod metric $\rho$ (see  \cite[Section 12]{bil}). Then $(\Omega,\rho)$ is a Polish space.
We also let
\[
X_t(\omega):=\omega(t),\quad \omega\in \Omega,\, t\ge 0,\quad \mathcal F^X_t:=\sigma(X_s: s\le t).
\]
Throughout the paper we adopt the convention that $f(\partial)=0$
for any function $f$ defined on a subset of $\mathbb R^d$.
By \cite[Theorem I.9.4]{BG} for any Feller semigroup $(T_t)$,
 there exists a unique family 
$(\mathbb P_x,\, x\in \mathbb R^d_\partial)$ of Borel probability measures
on $\Omega$ such that:
\begin{enumerate}
\item for any $f\in C_0(\mathbb R^d)$
\[
T_tf(x)=\mathbb E_xf(X_t),\quad x\in\mathbb R^d,\, t\ge 0,
\]
\item for all $x\in\mathbb R^d$, $s,t\ge 0$, we have  $\mathbb E_{X_t}f(X_s)=\mathbb E_x[f(X_{s+t})|\FF_t^X]$.
\end{enumerate}
The second condition says that $(\mathbb P_x,\, x\in \mathbb R^d_\partial)$ is a {\em Markov process} (also called a {\em Markov family}).  Any family $(\mathbb P_x,\, x\in \mathbb R^d_\partial)$ of Borel probability measures
on $\Omega$ satisfying (1)-(2) for a Feller semigroup $(T_t)$ is called a {\em Feller process}.
Here and in what follows we use the standard  notation of probability theory: 
 for any measurable function $Y$ on $\Omega$,
\[
\mathbb E_xY:=\int_{\Omega}Y(\omega)\,\mathbb P_x(d\omega).
\]
Furthermore, $\mathbb E_x[f(X_{s+t})|\FF_t^X]$ denotes the conditional expectation with respect to $\mathbb P_x$, equivalently, 
orthogonal projection of function $f(X_{s+t})\in L^2(\Omega,\FF^X_{s+t},\mathbb P_x)$ 
onto $L^2(\Omega,\FF^X_{t},\mathbb P_x)$ (see \cite[Section 1.3]{Kolokoltsov}).
In what follows by $(\mathbb P^\kappa_x,\, x\in \mathbb R^d_\partial)$, 
we denote the Feller process  related to $L_\kappa$.

Let $(\mathbb P^{*}_x,\, x\in \mathbb R^d_\partial)$ 
denote the Feller process related to the Feller generator $A^*-b\nabla$.
By \cite[Theorem 8.4]{Getoor} for any $\kappa\ge \kappa_0$,
\[
T_t^{\kappa,(p),*}f(x)=\mathbb E_x^{*}f(X_t)e^{\int_0^t(\text{div}b-\kappa)(X_s)\,ds}\quad a.e.
\]
By \cite[Corollary III.3.16]{BG} there exists a Hunt process $(\mathbb P_x^{\kappa,*},x\in\mathbb R^d_\partial)$ (for the definition see \cite[Definition I.9.2]{BG}) such that  
\[
\mathbb E_x^{*}f(X_t)e^{\int_0^t(\text{div}b-\kappa)(X_s)\,ds}=\mathbb E_x^{\kappa,*}f(X_t),\quad x\in\mathbb R^d,\, t>0,\, f\in\mathscr B_b(\mathbb R^d).
\]
In what follows for  non-negative functions $f\in \mathscr B^*(\mathbb R^d)$  we let
\[
P^\kappa_tf(x):=\mathbb E^\kappa_x f(X_t),\quad P^{\kappa,*}_tf(x):=\mathbb E^{\kappa,*}_x f(X_t),\qquad x\in \mathbb R^d.
\]
For any  $f\in\mathscr B^*(\mathbb R^d)$, we let
$P^\kappa_tf(x):= P^\kappa_tf^+(x)-P^\kappa_tf^-(x)$ if $P^\kappa_tf^+(x)<\infty$ or $P^\kappa_tf^-(x)<\infty$, and zero otherwise.
We adopt   analogous convention  for $P^{\kappa,*}_tf$. 
We also let for any  positive  $f\in\mathscr B^*(\mathbb R^d)$
\[
R^\kappa_\alpha f(x):= \int_0^\infty e^{-\alpha t} P^\kappa_t f(x)\,dt, \quad R^{\kappa,*}_\alpha f(x):= \int_0^\infty e^{-\alpha t} P^{\kappa,*}_t f(x)\,dt,\qquad x\in \mathbb R^d.
\]
For each   $R^\kappa_\alpha f$ and $R^{\kappa,*}_\alpha f$, with $f\in\mathscr B^*(\mathbb R^d)$, we adopt the convention 
analogous to that for  $(P^\kappa_tf), (P^{\kappa,*}_tf)$.
Finally, we denote  $R^\kappa:=R^\kappa_0, R^{\kappa,*}:=R^{\kappa,*}_0$. One easily checks  that
\begin{equation}
\label{eq.red1}
P^{\kappa+\alpha}_t=e^{-\alpha t}P^{\kappa}_t,\quad P^{\kappa+\alpha,*}_t= e^{-\alpha t}P^{\kappa,*}_t,\quad R^{\kappa+\alpha}_0=R^{\kappa}_\alpha,\quad R^{\kappa+\alpha,*}_0= R^{\kappa,*}_\alpha
\end{equation}
for any $\alpha,t\ge 0$.
Observe that for any $\varepsilon>0$   
\[
R^{\varepsilon}f(x)=\int_0^\infty e^{-t\varepsilon}P^{0}_tf(x)\,dt,\quad R^{\kappa_0+\varepsilon,*}f(x)=\int_0^\infty e^{-t\varepsilon}P^{\kappa_0,*}_tf(x)\,dt,\quad f\in C_0(\mathbb R^d).
\]
As a result 
\begin{equation}
\label{eq.bres}
R^\varepsilon(\mathscr B_b(\mathbb R^d))\subset \mathscr B_b(\mathbb R^d)\quad 
 R^{\kappa_0+\varepsilon,*}(\mathscr B_b(\mathbb R^d))\subset \mathscr B_b(\mathbb R^d).
\end{equation}

\begin{hypothesis}
\label{eq.stanass}
Throughout the paper we assume that for some $\kappa\ge \kappa_0,   (R^\kappa_\alpha)_{\alpha>0}, \,
(R^{\kappa,*}_\alpha)_{\alpha>0}$ are strongly Feller, i.e.  for any 
$f\in\mathscr B_b(\mathbb R^d)$ the functions  $R^\kappa_\alpha f, R^{\kappa,*}_\alpha f$  belong to  $C_b(\mathbb R^d)$  for each  $\alpha>0$.
\end{hypothesis}

Clearly, by the resolvent identity, $(R^\kappa_\alpha)_{\alpha>0}$ is strongly Feller for any $\kappa\ge 0$.
Furthermore, $R^\kappa_\alpha(x,dy)\ll dy$ for any $x\in\mathbb R^d$, $\alpha> 0$ and $\kappa\ge 0$.
Let $V\in\mathcal B(E)$ and $(S_t)$ be  a semigroup of  positive linear operators on $p\mathscr B^*(V)$.
Recall (see \cite[Section II]{BG}) that   a positive function $u\in \mathscr B^*(V)$
is called $\alpha$-excessive with respect to  $(S_t)$ provided that:
\begin{enumerate}
\item[a)] $e^{-\alpha t}S_tu(x)\le u(x),\, t\ge 0,\, x\in V$,
\item[b)] $\lim_{t\to 0^+}S_tu(x)=u(x),\, x\in V$.
\end{enumerate}
If $\alpha=0$ we simply say that $u$ is $(S_t)$-excessive.
By \cite[Theorem VI.1.4]{BG} for any $\alpha>0$ and $\kappa\ge\kappa_0$ there exists a  unique
function $G^\kappa_\alpha \in p\mathscr B(\mathbb R^d\times \mathbb R^d)$
such that for any $f\in p\mathscr B(\mathbb R^d)$,
\begin{equation}
\label{eq.GGG}
R^\kappa_\alpha f(x)=\int_{\mathbb R^d} G^\kappa_\alpha(x,y)f(y)\, dy,\quad R^{\kappa,*}_\alpha f(x)=\int_{\mathbb R^d} G^\kappa_\alpha(y,x)f(y)\, dy,\quad x,y\in\mathbb R^d,
\end{equation}
and  $G^\kappa_\alpha(\cdot,y)$, $G^\kappa_\alpha(x,\cdot)$ are  $\alpha$-excessive with respect to $(P^\kappa_t)$, $(P^{\kappa,*}_t)$, respectively
for any $x,y\in\mathbb R^d$.
The function $G^\kappa_\alpha(\cdot,\cdot)$ is called $\alpha$-Green's function for $-A+b\nabla+\kappa I$.

Fine topology (resp. co-fine topology) on $\mathbb R^d$ is the coarsest topology that makes all the $(P^\kappa_t)$-excessive (resp. $(P^{\kappa,*}_t)$-excessive) functions continuous. 
By \cite[Proposition 10.8]{Sharpe} a nearly Borel function $u:\mathbb R^d\to [0,\infty]$ is finely-continuous 
on  $\mathbb R^d$ if and only if   the process 
\begin{equation}
\label{eq.efq}
[0,\infty)\ni t\mapsto u(X_t)\in \overline {\mathbb R }\quad \text{ is right-continuous
at } t=0 \, \mathbb P_x\text{-a.s.},\, x\in\mathbb R^d. 
\end{equation}

For any Borel  set $B\subset\mathbb R^d$, we let 
\begin{equation}
\label{eq.hitt}
\tau_B(\omega):=\inf\{t>0: X_t(\omega)\notin B\}.
\end{equation}
For  open  set $V\subset\mathbb R^d$, we let 
 $(\mathbb P^{\kappa,V}_x,\, x\in V\cup\{\partial\})$, $(\mathbb P^{\kappa,*,V}_x,\, x\in V\cup\{\partial\})$
denote the Markov processes (see e.g.  \cite[Theorem III.3.3]{BG} applied to $\mathbb P_x$ and  $M_t:=\mathbf1_{[0,\tau_V)}(t)$)
 that are built  from $(\mathbb P^\kappa_x,\, x\in \mathbb R^d\cup\{\partial\})$, $(\mathbb P^{\kappa,*}_x,\, x\in \mathbb R^d\cup\{\partial\})$
by the {\em killing upon the exit of $V$}, i.e. 
\[
\mathbb P^{\kappa,V}_x:= (k_V)_{\sharp} \mathbb P^\kappa_x,\quad \mathbb P^{\kappa,*,V}_x:= (k_V)_\sharp \mathbb P^{\kappa,*}_x,
\]
- push-forwards of measures $\mathbb P^\kappa_x, \mathbb P^{\kappa,*}_x$ through $k_V$,  where $k_V:\Omega\to \Omega$ is defined as 
\[
k_V(\omega)(t)=\omega(t)\mathbf1_{\{t<\tau_V(\omega)\}}+\partial \mathbf1_{\{t\ge\tau_V(\omega)\}}.
\]
It is known  (see e.g.  \cite[Theorem III.3.3]{BG} applied to $\mathbb P_x$ and  $M_t:=\mathbf1_{[0,\tau_V)}(t)$ again) that   
\begin{equation}
\label{eq.pdef}
\begin{split}
&P^{\kappa,V}_tf(x):=\mathbb E^{\kappa,V}_x [f(X_t)]=\mathbb E^\kappa_x [f(X_t)\mathbf 1_{\{t<\tau_V\}}]
=\mathbb E^0_x [e^{-\kappa t}f(X_t)\mathbf 1_{\{t<\tau_V\}}],
\\&P^{\kappa,*,V}_tf(x):=\mathbb E^{\kappa,*,V}_x [f(X_t)]=\mathbb E^{\kappa,*}_x [f(X_t)\mathbf 1_{\{t<\tau_V\}}]
=\mathbb E^{\kappa_0,*}_x [e^{-(\kappa-\kappa_0) t}f(X_t)\mathbf 1_{\{t<\tau_V\}}]
\end{split}
\end{equation}
define Markov semigroups.
We let
\begin{equation}
\label{eq.rdef}
R^{\kappa,V}_\alpha f(x):=\mathbb E^\kappa_x \int_0^{\tau_V} e^{-\alpha t}f(X_t)\, dt,\quad
R^{\kappa,*,V}_\alpha f(x):=\mathbb E^{\kappa,*}_x \int_0^{\tau_V} e^{-\alpha t}f(X_t)\, dt,\quad x \in V.
\end{equation}
Observe that 
\[
R^{\kappa,V}_\alpha f(x)= \int_0^\infty e^{-\alpha t}P^{\kappa,V}_tf(x)\,dt,\quad R^{\kappa,*,V}_\alpha f(x)=\int_0^\infty e^{-\alpha t}P^{\kappa,*,V}_tf(x)\,dt,\quad x\in V.
\]
One easily checks again (by using \eqref{eq.pdef}) that
\begin{equation}
\label{eq.red2}
P^{\kappa+\alpha,V}_t=e^{-\alpha t}P^{\kappa,V}_t,\quad P^{\kappa+\alpha,*,V}_t= e^{-\alpha t}P^{\kappa,*,V}_t,\quad R^{\kappa+\alpha,V}_0=R^{\kappa,V}_\alpha,\quad R^{\kappa+\alpha,*,V}_0= R^{\kappa,*,V}_\alpha
\end{equation}
for any $\alpha,t\ge 0$.
Notice that for  $f\in p\mathscr B(\mathbb R^d)$, $P^{\kappa,V}_tf\le P^\kappa_tf$ for all $t>0$.
One easily deduces then (see \eqref{eq.lp}) that  for any $p>1$, $(P^{\kappa,V}_t), (P^{\kappa,*,V}_t)$ 
can be extended from $C_b(V)\cap L^p(V)$ to  strongly continuous semigroups on $L^p(V)$, 
that we shall denote by  $(T^{\kappa,(p),V}_t), (T^{\kappa,(p),*,V}_t)$, respectively.
Furthermore   both semigroups are   submarkovian, hence, by the interpolation theorem again, they are contractions on $L^p(V)$.
We denote by $L^{(p)}_{\kappa,V}, L^{(p),*}_{\kappa,V}$ respective generators of  $(T^{\kappa,(p),V}_t), (T^{\kappa,(p),*,V}_t)$.
We further define operators
\[
\Pi^\kappa_V,\,\Pi^{\kappa,*}_V : \mathscr B_b(\mathbb R^d)\to \mathscr B_b(\mathbb R^d)
\]
by
\[
\Pi^\kappa_V(f)(x):= \mathbb E^\kappa_x f(X_{\tau_V}),\quad \Pi^{\kappa,*}_V(f)(x):= \mathbb E^{\kappa,*}_x f(X_{\tau_V}),\quad x\in\mathbb R^d.
\]
By  \cite[Theorem III.3.3]{BG} applied to $\mathbb P_x$ and  $M_t:=e^{-\kappa t}$ one has
\begin{equation}
\label{eq2.killex}
 \mathbb E^\kappa_x f(X_{\tau_V})= \mathbb E^0_x e^{-\kappa\tau_V}f(X_{\tau_V}),\quad x\in\mathbb R^d.
\end{equation}

It is an elementary calculation  (Dynkin's formula; see e.g. \cite[Theorem 12.16]{Sharpe}) that 
\[
R^{\kappa,V}f(x)=R^\kappa f(x)-\Pi^\kappa_V(R^\kappa f)(x),\quad 
R^{\kappa,*,V}f(x)=R^{\kappa,*}f(x)-\Pi_V^{\kappa,*}(R^{\kappa,*}f)(x),\quad x\in\mathbb R^d.
\]
By \cite[Theorem 3.4]{SW1} $(R^{\kappa,V}_\alpha)_{\alpha\ge 0}$  is strongly Feller provided that $V$
is open. 
By \cite[Theorem VI.1.16]{BG}
\[
\int_{\mathbb R^d}g \Pi^\kappa_VR^\kappa f=\int_{\mathbb R^d} f\,\Pi^{\kappa,*}_VR^{\kappa,*}g.
\]
This combined with the previous equations yields
\[
\int_{\mathbb R^d} R^{\kappa,V}f\,g=\int_{\mathbb R^d} f\,R^{\kappa,*,V}g,\quad f,g\in\mathscr B_b(\mathbb R^d).
\]
By \cite[Theorem VI.1.4]{BG},  for any $\alpha>0$ and $\kappa\ge \kappa_0$ there exists $\alpha$-Green's function 
\begin{equation}
\label{eq.gvvve}
G^\kappa_{V,\alpha}:V\times V\to [0,\infty],
\end{equation}
related to  $(R^{\kappa,V}_\alpha)_{\alpha\ge 0}$, and
according to \cite[Definition VI.1.2]{BG}   $(R^{\kappa,V}_\alpha)$ and $(R^{\kappa,*,V}_\alpha)$ 
are {\em in duality relative to the Lebesgue measure}.
The role of $u_\alpha(x,y)$ considered in \cite[Definition VI.1.2]{BG}
plays here $G^\kappa_{V,\alpha}(x,y)$.
From this and \eqref{eq.bres} we obtain \cite[(VI.2.1),(VI.2.2)]{BG}. 
Consequently, the standing assumptions of \cite[Section VI]{BG} are satisfied
and we may freely use the results contained therein. 
 
We set 
\[
G^\kappa_{V}(x,y)=G^\kappa_{V,0}(x,y):=\sup_{\alpha>0}G^\kappa_{V,\alpha}(x,y)=\lim_{\alpha\to 0^+} G^\kappa_{V,\alpha}(x,y).
\]
By \eqref{eq.bres}, $G^\kappa_{V}$ is finite a.e. whenever $\kappa>0$.

For any Borel positive measure $\mu$ on $V$ and $\alpha\ge 0$ we let
\begin{equation}
\label{dual.res2}
R^{\kappa,V}_\alpha\mu(x):=\int_{V}G^\kappa_{V,\alpha}(x,y)\,\mu(dy),\quad  R^{\kappa,*,V}_\alpha \mu(x):=\int_{V}G^\kappa_{V,\alpha}(y,x)\,\mu(dy),\quad x\in V.
\end{equation}
We also use the shorthand notation $R^{\kappa,V}\mu=R^{\kappa,V}_0\mu$, $R^{\kappa,*,V} =R^{\kappa,*,V}_0$.
Throughout the paper we adopt the convention that $R^{\kappa,V}_\alpha\mu$, for any positive Borel measure $\mu$ on $V$,
is a function on $\mathbb R^d$ by letting $R^{\kappa,V}_\alpha\mu(x)=0,\, x\in\mathbb R^d\setminus V$.
For any positive Borel functions $f,g$ on $\mathbb R^d$ and positive Borel measure $\mu$ on $\mathbb R^d$ we let
\[
(f,g):=\int_{\mathbb R^d}fg,\quad (\mu,f)=\int_{\mathbb R^d}f\,d\mu.
\]

\begin{lemma}
\label{lm.duoper}
Let  $\kappa>\kappa_0$ and $\mu$ be a positive bounded Borel measure  on $D$.
Then  $R^{\kappa,D}\mu\in L^1(D)$ and  for any $\eta\in C_c^2(D)$,
\[
(R^{\kappa,D}\mu,L^{*}_\kappa\eta)=(\mu,\eta).
\]
\end{lemma}
\begin{proof}
As to the first assertion, observe that by  \eqref{dual.res2} and \eqref{eq.red1}
\[
\int_DR^{\kappa,D}\mu(x)\,dx=(R^{\kappa,D}\mu,1)=(\mu,R^{\kappa,*,D}1)\le \frac{\mu(D)}{(\kappa-\kappa_0)}.
\]
For the second assertion, we first note that $L^{*}_\kappa\eta\in C_b(\mathbb R^d)$ and $L^{*}_\kappa\eta=L^{(2),*}_{\kappa,D}\eta$.
Therefore, by \eqref{dual.res2} again,
\[
(R^{\kappa,D}\mu,L^{*}_\kappa\eta)=(\mu,R^{\kappa,*,D}L^{*}_\kappa\eta)=(\mu,R^{\kappa,(2),*,D}L^{(2),*}_{\kappa,D}\eta)=(\mu,\eta),
\]
which completes the proof.
\end{proof}

\subsection{Polar sets}
\label{s.ps}
Recall that  $R^\kappa,\, \kappa\ge 0$ and $R^{\kappa,*},\, \kappa\ge \kappa_0$ are strongly Feller.
In probabilistic potential theory a Borel measurable set $B\subset\mathbb R^d$ is called {\em polar},
with respect to $(\mathbb P^\kappa_x)$,
provided that 
\begin{equation}
\label{eq.polar}
\mathbb P^\kappa_x(\sigma_B<\infty)=0,\quad x\in\mathbb R^d,
\end{equation}
where $\sigma_B:=\inf\{t>0: X_t\in B\}$.
By \cite[Proposition VI.4.3]{BG}  for any bounded  Borel set $B\subset\mathbb R^d$, $\kappa\ge \kappa_0$, and  $x\in\mathbb R^d$,
\[
\mathbb P^\kappa_x(\sigma_B<\infty)=\sup_{\mu\in\mathcal M^+_1(B)}R^\kappa\mu(x),
\]
where $\mathcal M^+_{\kappa,1}(B):=\{\mu: R^\kappa\mu\le 1 \text{ in }\mathbb R^d,\, \mu\text{ is a positive Borel measure supported in }B\}$.
Furthermore, there exists a measure $\mu_B\in \mathcal M^+_{\kappa,1}(B)$ such that $\mathbb P^\kappa_x(\sigma_B<\infty)=R^\kappa\mu_B$ in $\mathbb R^d$.
The same notion can be introduced for $(\mathbb P^{\kappa,*}_x)$, with $\kappa\ge \kappa_0$.
We denote the objects defined above, but with $(\mathbb P^\kappa_x)$ replaced by $(\mathbb P^{\kappa,*}_x)$,
by adding    superscript $"^*"$ to the respective expressions.

Now define for $B\in\mathcal B(\mathbb R^d)$ and $\kappa\ge 0$ ($\kappa\ge \kappa_0$ in case of dual notion),
\[
C_\kappa(B):=\sup_{\mu \in\mathcal M^+_{\kappa,1}(B)}\mu(\mathbb R^d),\quad C^*_\kappa(B):=\sup_{\mu \in\mathcal M^{+,*}_{\kappa,1}(B)}\mu(\mathbb R^d).
\]
Then a  Borel set $B\subset\mathbb R^d$ is polar with respect to $(\mathbb P^\kappa_x)$ (resp. $(\mathbb P^{\kappa,*}_x)$)
if and only if $C_\kappa(B)=0$ (resp. $C^*_\kappa(B)=0$). By Proposition \cite[Proposition VI.4.4]{BG}
\begin{equation}
\label{eq.eqpol}
B\in\mathcal B(\mathbb R^d) \text{ is polar with respect to } (\mathbb P^\kappa_x) \text{ if and only if it is  polar with respect to } (\mathbb P^{\kappa,*}_x).
\end{equation}
Finally, by \eqref{eq2.killex}, for any $\kappa\ge 0$
\[
\mathbb P^\kappa_x(\sigma_B<\infty)=\lim_{\alpha\to 0^+}\mathbb E^\kappa_x e^{-\alpha\sigma_B}
=\lim_{\alpha\to 0^+}\mathbb E^0_x e^{-(\alpha+\kappa)\sigma_B}=\mathbb E^0_x e^{-\kappa\sigma_B},\quad x\in\mathbb R^d.
\]
This shows that polarity of a set does not depend on $\kappa\ge 0$, i.e.
\[
\text{for some }\,\kappa\ge 0,\,\, \mathbb P^\kappa_x(\sigma_B<\infty)=0\quad\Leftrightarrow \quad
\text{for any }\,\kappa\ge 0,\,\, \mathbb P^\kappa_x(\sigma_B<\infty)=0.
\]
On the other hand, by the resolvent identity, one easily deduces that
\[
C_{\kappa}(B)=0\text{ for some } \kappa\ge 0 \quad\Leftrightarrow \quad C_{\kappa}(B)=0\text{ for any } \kappa\ge 0.
\]

Consequently, we introduce the following definition.

\begin{definition}
A  set $B\in \mathcal B(\mathbb R^d)$ is called {\em polar with respect to the operator } $A-b\nabla$ if
$C_{1}(B)=0$, equivalently,
\[
\mathbb P^0_x(\sigma_B<\infty)=0,\quad x\in\mathbb R^d.
\]
\end{definition}

\subsection{Fine topology and co-fine topology}

Recall that the fine topology (with respect to $(P^\kappa_t)$) on $\mathbb R^d$ 
is the coarsest topology that makes all the $(P^\kappa_t)$-excessive functions continuous. 
This is an intrinsic topology depending on the underlying   Markov semigroup.
However, the property \eqref{eq.efq} and further identities \eqref{eq.pdef}, show 
that the fine topology is independent of $\kappa\ge 0$ and the fine topology restricted to a Borel measurable finely open
set $V\subset\mathbb R^d$ agrees  with the one generated by $(P_t^{\kappa,V})$ for any $\kappa\ge 0$.
In other words, denoting by $\mathcal O_\kappa$ the fine topology generated by $(P^\kappa_t)$, we have
\[
\mathcal O_0=\mathcal O_\kappa,\quad \kappa\ge 0,
\]
and denoting  by $\mathcal O^V_\kappa$ the fine topology generated by $(P^{\kappa,V}_t)$, with $V\in\mathcal B(\mathbb R^d)
\cap  \mathcal O_0$, we have 
\[
\mathcal O^V_\kappa=\mathcal O_0\cap V,\quad \kappa\ge 0.
\]
Consequently, in the whole paper, 
we use the notion of fine continuity and fine topology with respect  to $(P^0_t)$ 
without specifying the semigroup in the notation.

Throughout the paper we shall frequently use the following  well known fact (see e.g. \cite[Exercise 10.25]{Sharpe}).
Let $V$ be a finely open Borel subset of $\mathbb R^d$.
\begin{equation}
\label{eq.ppp1}
\text{If } u,v\in \mathscr B(\mathbb R^d) \text { are finely continuous on }V \text{ and }u=v \text{ a.e. then } u(x)=v(x),\, x\in V.
\end{equation}
All the  comments and conclusions of this section hold true for co-fine topology 
(bearing in mind, however, that $(P^{\kappa,*}_t)$ is defined for $\kappa\ge \kappa_0$).
Applying \eqref{eq.ppp1} and \eqref{eq.red2} we obtain 
\begin{equation}
\label{eq.greenred}
G^\kappa_{V,\alpha}(x,y)=G^{\kappa+\alpha}_{V,0}(x,y),\quad x,y\in V
\end{equation}
for any $\alpha\ge 0,\, \kappa\ge \kappa_0$.

\section{Examples}

In this short section we shall give a few examples of concrete operators which the results of the present paper 
may be applied to.  

\subsection{Operators satisfying Hypothesis \ref{eq.stanass}}

\begin{example}[Fractional Laplacian]
\label{ex11}
Let  $A=\Delta^s$, i.e. $\psi(\xi)=|\xi|^{2s}$, for some $s\in (0,1)$,
then  $A-b\nabla$ satisfies Hypothesis \ref{eq.stanass} (see \cite[Theorem 2.2]{KK}). 
\end{example}

\begin{example}[Cylindrical L\'evy noise]
\label{ex12}
Let  $A=\sum_{i=1}^d(\frac{\partial^2}{\partial x_i^2})^{s_i}$  for some $s_i\in (0,1)$.
Then  $A-b\nabla$ satisfies Hypothesis \ref{eq.stanass} (see \cite[Theorem 1]{KPP}). Note that such operators 
have L\'evy measure singular with respect to the Lebesgue measure.
\end{example}

\begin{example}[H\"ormander's condition]
\label{ex13}
Let $A$ be a L\'evy operator with triplet  $(0,Q, C_{\sharp}\nu)$, where
\[
\nu(dz)=\frac{a(z)}{|z|^{d+2s}}\,dz
\]
for some $s\in (0,1)$, $C$ is a $d\times d$ matrix,  and $a\in C^1_b(\mathbb R^d)$ satisfies $a(z)=a(-z)$, i.e.
\[
Au(x)=\frac12\text{Tr}(QD^2 u(x))+\int_{\mathbb R^d}(u(x+Cy)-u(x)-\langle Cy,\nabla u(x)\rangle\mathbf1_{B(0,1)})\,\nu(dy).
\]
Let $B_0:=Id$ be the identity matrix  and define 
$d \times d$-matrix-valued function $B_n(x)$ recursively by
\[
B_n^{ij}(x):=b(x) \cdot \nabla B^{ij}_{n-1}(x)-\nabla b^i(x) \cdot B^{j}_{n-1}(x)+\frac{1}{2}\text{Tr}(QD^2B^{ij}_{n-1}(x)),\quad n\ge 1,
\]
where $B^j_n$ denotes $j$-th column of $B_n$.
Assume that $b$ is smooth with bounded derivatives of all orders. 
If  for any $x \in \mathbb{R}^d$ there exists $n=n(x) \in \mathbb{N}$ such that 
$$
\operatorname{Rank}\left[\sqrt{Q}, B_1(x) \sqrt{Q}, \cdots, B_n(x) \sqrt{Q}, C, B_1(x) C, \cdots, B_n(x) C\right]=d,
$$
then $A-b\nabla$ satisfies Hypothesis \ref{eq.stanass}  (see \cite[Theorem 1.1]{SZ}).
\end{example}

\begin{example}(Mixed L\'evy operators)
\label{ex14}
Let $\phi$ be a Bernstein functions and $A:= -\phi(-\Delta)$. 
If 
\[
\liminf _{u \rightarrow \infty} \phi(u) u^{-s}>0
\]
for some $s \in(0,1)$, then $A-b\nabla$ satisfies Hypothesis \ref{eq.stanass}  (see \cite[Theorem 1.1]{WXZ}).
In particular one can take $A=\Delta^{s_1}+\Delta^{s_2}$ with $s_1\in (0,1], s_2\in (0,1)$.
\end{example}

\subsection{Polar     sets: some criterions}
The fact whether or not a set $K$ is polar  must be examined  on the case-by-case basis
depending on the structure of underlying  operator. Here we refer to some known results and briefly describe some possible  direct methods.

First of all  notice that a   consequence of the   results recalled  in  Section \ref{s.ps}  is that,
under Hypothesis \ref{eq.stanass},  a Borel set $B\subset \mathbb R^d$ 
is polar if and only if  $R^{\kappa}\mu$ is unbounded 
whenever $\mu$ is a bounded non-trivial positive  measure with compact support contained in $B$
(see the comments in the first paragraph following \cite[Proposition VI.4.3]{BG}).
In particular 
a point $\{x\}$ is polar if and only if $G^{\kappa}(x,x)=\infty$. As a result, lower and upper on-diagonal bounds on the 
heat kernels may decide whether a point is a polar set or not. Consequently,
using \cite[Theorem 2.2]{KK} yields that  $\{x\}$ is polar with respect to the operator of Example \ref{ex11}
 for any  $s\in (0,1)$ and $d\ge 2$ or $d=1$ and $s\in (0,1/2)$.
For the operators from  Examples \ref{ex12}--\ref{ex14} the reach literature 
on heat kernels estimate may by utilized (see e.g.  \cite[Theorem 1]{Ishikawa},  \cite[Theorem 1]{Picard}, \cite[Theorem 1.2]{BSK}, \cite[Theorem 3.4]{KKS}).  In order to treat other than singletons compact sets it is useful to have local estimates
on Green's functions and then use intrinsic Hausdorff measure technique (see \cite{HN} and  \cite[Theorem 1.2]{BSK} again).

A probabilistic approach on the other hand gives information on polarity of a compact set $K\subset\mathbb R^d$,
with respect to $A-b\nabla$,
without recourse to  asymptotics of related  heat kernels.  It is based on the law of iterated logarithm
for L\'evy processes. The reasoning below shows that  for the operators in  Examples \ref{ex12},\ref{ex14}
singletons are always polar. 

First, by Hypothesis \ref{eq.stanass}, we know that $\ell^d(A)=0$ implies $R^{\kappa}(\mathbf1_A)(x)=0$
for any $x\in\mathbb R^d$. Thus for any $f\in p\mathscr B(\mathbb R^d)$, by \cite[Proposition II.3.5]{BG}, if $\ell^d(R^\kappa f=\infty)=0$, then $B:=\{R^\kappa f=\infty\}$ is polar.  Further, observe that for $\kappa\ge \kappa_0+1$,
\[
\int_{\mathbb R^d}R^{\kappa}f=\int_{\mathbb R^d}fR^{\kappa,*}1\le  \int_{\mathbb R^d}f.
\]
Consequently, in order to prove that $K$ is polar it is enough to find $f\in p\mathscr B(\mathbb R^d)\cap L^1$
such that $K\subset \{R^\kappa f=\infty\}$. 
Let us note that  $\mathbb P_x^0$ is the distribution of a unique solution to 
\[
X^x_t=x+\int_0^tb(X^x_s)\,ds+L_t,
\]
where $(L_t)$ is a L\'evy process with symbol $\psi$.
Suppose that $Q\equiv 0$ and $\nu$ is symmetric.  Let 
\[
\beta_{L}:= \inf\{\alpha>0: \int_{B_1}|x|^\alpha\,\nu(dx)<\infty\}
\]
(so called Blumenthal-Getoor index, see \cite{BG1}). Observe that in Example \ref{ex14} $\beta_L\ge 2s$
and in Example \ref{ex12} $\beta_L=\min\{s_1,\dots,s_d\}$. By \cite[Theorem 3.1]{BG1}
\begin{equation}
\label{eq.lil}
\lim_{t\to 0} \frac{L_t}{t^{1/\alpha}}=0
\end{equation}
for any $\alpha>\beta_L$. Let $x_0\in\mathbb R^d$,
$d\ge 2$,  $\gamma\in (\beta_L\vee 1, d)$.  Set
\[
f(x):=\frac{1}{|x-x_0|^\gamma}\mathbf1_{B_1(x_0)}(x),\quad x\in\mathbb R^d.
\]
Clearly, $f\in L^1$.  Let $\tau:=\inf\{t>0: X^{x_0}_t\notin B_1(x_0)\}$.
Then 
\[
\int_0^\tau f(X^{x_0}_t)\,dt\ge \int_0^\tau \frac{c}{t^\gamma+|L_t|^\gamma}\,dt.
\]
By \eqref{eq.lil} for any $\omega\in\Omega$ and $\theta>1$ there exist $c_\omega,t_\omega>0$ such that
\[
|L_t(\omega)|\le c_\omega t^{\frac{1}{\theta \beta }},\quad t\le t_\omega.
\]
Consequently, letting $\hat t_\omega:= \min\{t_\omega,\tau(\omega),1\}$ and 
$\hat c_\omega:=(1+ c_\omega)^{-1}$,  and taking $\theta=\gamma/\beta_L$, we find 
\[
\int_0^{ t_\omega }  f(X^{x_0}_t(\omega))\,dt\ge
\hat c_\omega \int_0^{\hat t_\omega } \frac{c }{t^\gamma+ t^{\frac{\gamma}{\theta \beta}}}\,dt
=\hat c_\omega \int_0^{\hat t_\omega } \frac{c }{t^\gamma+ t}\,dt=\infty.
\]
As a result $R^\kappa f(x_0)=\infty$,
which shows that $\{x_0\}$ is polar.

Analogous reasoning may be proceeded for a $(d-2)$-dimensional (or lower)  hyperplane $H_{d-2}\subset \mathbb R^d$
(here $d\ge 3$). 
For simplicity, suppose that $H_{d-2}=\{x\in\mathbb R^d: \Pi(x)=0\}$, where $\Pi(x)=(x_1,x_2,0,\dots,0)$.
Set $f(x):= |\Pi(x)|^{-\gamma}\mathbf1_{B_1}(x)$, with $\gamma\in (\beta_{\Pi(L)}\vee 1,2)$. Noticing that $\Pi(L)$
is again a L\'evy process, $\beta_{\Pi(L)}\le \beta_L$ and $f\in  L^1$, one easily  computes that
$H_{d-2}$ is a polar set for operators from Examples \ref{ex11},\ref{ex12},\ref{ex14}. 

\section{Auxiliary results, excessive functions}
\label{sec4aux}
Throughout the section we fix $\kappa_1>\kappa_0$ and  we consequently drop it in the notation.
Namely, we let 
\begin{equation}
\label{eq.ozn}
\begin{split}
&P_t^V:=P^{\kappa_1,V}_t,\quad  P_t^{*,V}:=P^{\kappa_1,*,V}_t,\quad R^V_\alpha:=R^{\kappa_1,V}_\alpha,\quad
R^{*,V}_\alpha:=R^{\kappa_1,*,V}_\alpha,\quad
\mathbb E_x:=\mathbb E^{\kappa_1}_x,\\&\mathbb E_x^V:=\mathbb E^{\kappa_1,V}_x,\quad
\mathbb E_x^*:=\mathbb E^{\kappa_1,*}_x,\quad \mathbb E_x^{*,V}:=\mathbb E^{\kappa_1,*,V}_x, L:=L^\kappa_1,\quad L^*:=L^{\kappa_1,*},\\& L^{(p)}_V:= L^{(p)}_{\kappa_1,V},\quad  L^{(p),*}_V:= L^{(p),*}_{\kappa_1,V},\quad T_t^{(p),V}:=T_t^{\kappa_1,(p),V},\quad
T_t^{(p),*,V}:=T_t^{\kappa_1,(p),*,V}.
\end{split}
\end{equation}

We shall start with the following result that 
can be found in  \cite[Lemma II.1]{DL} or, under more restrictive assumptions on $b$, in  \cite[Lemma 5.6]{LLWX}.
The term $j_\varepsilon\ast(b\cdot\nabla u)$ below is understood as a distribution on $D$:
\[
(j_\varepsilon\ast(b\cdot\nabla u),\eta):= -(u,\text{div}[b(j_\varepsilon\ast\eta)]),\quad \eta\in C_c^\infty(D).
\]
Thanks to the assumptions made on $b$ one  have that  the distribution $j_\varepsilon\ast(b\cdot\nabla u)$
is in fact a locally integrable function.

\begin{lemma}
\label{lm2.1}
For any $u\in L^1_{loc}(D)$
\[
b\cdot\nabla (u\ast j_\varepsilon)-j_\varepsilon\ast(b\cdot\nabla u)\to 0\quad\text{a.e. in }D,
\]
and in $L^1_{loc}(D)$.
\end{lemma}

By the very definition $G_V(\cdot,y)$ is $(P^V_t)$-excessive for any $y\in V$ and 
$G_V(x,\cdot)$ is $(P^{V,*}_t)$-excessive for any $x\in V$. Therefore $R^V\mu$ (resp. $R^{V,*}\mu$) is $(P^V_t)$-excessive 
(resp. $(P^{V,*}_t)$-excessive) for any positive Borel measure $\mu$ on $V$. Note also that under assumptions that we made in 
the previous section any $(P^V_t)$-excessive function is Borel measurable (see \cite[Proposition V.1.4]{BG}).

\begin{lemma}
\label{lm2.2}
Let $V$ be an open bounded subset of $\mathbb R^d$ and $(\beta_n), \beta$ be  bounded 
positive Borel measures on $V$. Suppose that   $\beta_n\to \beta $ in $[C_b(V)]^*$, then
\begin{enumerate}
\item[(i)] $R^V\beta_n+R^V\beta<\infty$ q.e., $n\ge 1$,
\item[(ii)] $R^V\beta_n\to R^V\beta$ a.e. up to a subsequence.
\end{enumerate}
\end{lemma}
\begin{proof}
(i) Let $g$ be a bounded strictly positive Borel function on $V$. Then
\[
(R^V\beta,g)=(\beta,R^{*,V}g)\le \|R^{*,V}g\|_\infty\beta(V)\le \frac{1}{\kappa_1-\kappa_0}\|g\|_\infty\beta(V)<\infty,
\]
and the same holds true for $\beta_n$. Now, the result follows from \cite[Proposition VI.2.3]{BG}.

(ii) By \cite[Lemma 94, page 306]{DellacherieMeyer} there exists a subsequence (not relabeled) 
and a $(P^V_t)$-excessive function $w$  such that $R^V\beta_n\to w$ a.e. 
Let $\xi\in \mathscr B_b(V)$. By Hypothesis \ref{eq.stanass} $R^{*,V}\xi\in C_b(V)$ (see also \cite[Theorem 3.4]{SW1}).
Thus 
\[
(R^{V}\beta_n,\xi)=(\beta_n,R^{*,V}\xi)\to (\beta,R^{*,V}\xi)=(R^{V}\beta,\xi).
\]
In other words, $R^{V}\beta_n$ weakly converges to $R^V\beta$ in $L^1(V)$.
By the Dunford-Pettis theorem the sequence $(R^{V}\beta_n)_{n\ge 1}$ is uniformly integrable.
As a result, by Vitali's theorem,  $w=R^V\beta$.
\end{proof}

\begin{lemma}
\label{lm2.3.1}
Let  $(V_n)$ be an increasing sequence of open subsets of $\mathbb R^d$ and let $V:=\bigcup_{n\ge 1} V_n $.
Suppose that $(w_n)$ is an increasing sequence of positive Borel functions on $\mathbb R^d$
and $w_n$ is $(P^{V_n}_t)$-excessive for each $n\ge 1$.
Then $w:=\lim_{n\to \infty}w_n$ is $(P^V_t)$-excessive.
\end{lemma}
\begin{proof}
By \cite[Proposition II.2.2]{BG} $w$ is $(P^{V_{n}}_t)$-excessive for any $n\ge 1$.  As a result $w$
is finely continuous on $V$ and (see \cite[Proposition II.2.3]{BG})
\[
\alpha R^{V_n}_\alpha w(x) \le w(x),\quad x\in V_n,\,n\ge 1.
\]
We have (see \eqref{eq.rdef})
\[
R^{V_n}_\alpha w(x)=\mathbb E_x \int_0^{\tau_{V_n}} e^{-\alpha t}w(X_s)\,ds,\quad x\in V_n.
\]
On the other hand one easily sees that
\[
\tau_{V_n}\nearrow \tau_V,\quad n\to \infty.
\]
Therefore, letting $n\to \infty$ in the above inequality yields
\[
\alpha R^V_\alpha w(x)\le w(x),\quad x\in V
\]
for any $\alpha>0$. By fine continuity of $w$ on $V$  we get that $\alpha R^V_\alpha w(x)\to w(x),\, x\in V$
as $\alpha\searrow 0$. Thus, $w$ is $(P^V_t)$-excessive (see  \cite[Proposition II.2.3]{BG} again).
\end{proof}

\begin{lemma}
\label{lm2.3}
Let  $(V_n)$ be an increasing sequence of  open  subsets of $\mathbb R^d$ and let $V:=\bigcup_{n\ge 1} V_n $.
Let $\beta$ be a positive Borel measure on $V$. Then
\[
R^{V_n}\beta(x)\nearrow R^V\beta (x),\quad x\in V.
\]
\end{lemma}
\begin{proof}
By \eqref{eq.rdef} for any positive Borel function $\eta$ on $V$
\[
R^{*,V_n}\eta(x)=\mathbb E_x^* \int_0^{\tau_{V_n}} \eta(X_s)\,ds,\quad x\in V_n.
\]
Thus, 
\[
R^{*,V_n}\eta\le R^{*,V_{n+1}}\eta\le R^{*,V}\eta \text{ on } V \text{ for any } n\ge 1.
\]
Consequently,
\[
(R^{V_n}\beta,\eta)=(\beta,R^{*,V_n}\eta)\le (\beta,R^{*,V_{n+1}}\eta)=(R^{V_{n+1}}\beta,\eta),\quad\eta\in\mathscr B^+(\mathbb R^d).
\]
The same clearly holds true with $V_{n+1}$ replaced by $V$. By the choice of $\eta$
we  infer that 
\begin{equation}
\label{eq.2nn}
R^{V_n}\beta\le R^{V_{n+1}}\beta\le R^{V}\beta \quad\text{a.e.  on } V \text{ for any } n\ge 1.
\end{equation}

Note that  for any $m\ge 1$, each of functions $R^{V_n}\beta,\,n\ge m $ and $R^{V}\beta$ is $(P^{V_{m}}_t)$-excessive.
Therefore, by \eqref{eq.ppp1} the inequalities in  \eqref{eq.2nn} hold for any $x\in V$.
Set $w(x):=\lim_{n\to \infty} R^{V_n}\beta(x),\, x\in\mathbb R^d$. 
By \cite[Proposition II.2.2]{BG} $w$ is $(P^{V_{n}}_t)$-excessive for any $n\ge 1$.  As a result $w$
is finely continuous on $V$. Obviously, $w\le R^V\beta$.
Observe that
\[
\tau_{V_n}\nearrow \tau_V,\quad n\to \infty.
\]
As a result $R^{*,V_n}\eta\nearrow R^{*,V}\eta$. Therefore
\[
(w,\eta)=\lim_{n\to \infty} ( R^{V_n}\beta,\eta)=\lim_{n\to \infty} (\beta,R^{*,V_n}\eta)=(\beta,R^{*,V}\eta)=(R^{V}\beta,\eta).
\]
Hence, $w=R^V\beta$ a.e., which in turn implies, by \eqref{eq.ppp1} again, that $w=R^V\beta$ in $V$.
\end{proof}

\begin{lemma}
\label{lm2.4}
Let $V$ be an    open   subset of $\mathbb R^d$ and $B\in\mathcal B(\mathbb R^d)$ be a   polar closed subset of $V$.
Furthermore, let $\beta$ be a positive Borel measure on $V$. Then
\[
R^V\beta(x)=R^{V\setminus B}\beta(x),\quad x\in V\setminus B.
\]
\end{lemma}
\begin{proof}
Suppose that we know that 
\begin{equation}
\label{eq.dualid}
R^{*,V}\eta(x)=R^{*,V\setminus B}\eta(x),\quad x\in V\setminus B
\end{equation}
for any positive Borel function $\eta$ on $\mathbb R^d$.
Then
\[
(R^V\beta,\eta)=(\beta,R^{*,V}\eta)=(\beta,R^{*,V\setminus B}\eta)=(R^{V\setminus B} \beta,\eta).
\]
Consequently, $R^V\beta=R^{V\setminus B} \beta$ a.e. Since both functions are $(P_t^{V\setminus B})$-excessive
we have, by  \eqref{eq.ppp1}, that $R^V\beta=R^{V\setminus B} \beta$ on $V\setminus B$. What is left is to show \eqref{eq.dualid}. But (see \eqref{eq.rdef})
\[
R^{*,V}\eta(x)=\mathbb E^*_x\int_0^{\tau_V}\eta(X_s)\,ds,\quad R^{*,V\setminus B}\eta(x)=\mathbb E^*_x\int_0^{\tau_{V\setminus B}}\eta(X_s)\,ds,\quad x\in V\setminus B.
\]
Since $B$ is polar, we have $\tau_{V}=\tau_{V\setminus B}\, \mathbb P^*_x$-a.s.
for any $x\in\mathbb R^d$ (see \eqref{eq.eqpol}), which implies the desired equality.
\end{proof}

We close this section by proving a slight generalization of the well known Dynkin formula.

\begin{lemma}(Dynkin formula)
\label{lm.dynmeas}
Let $B,V$ be  open subsets of $\mathbb R^d$ and  $B\subset V$. 
Then for any positive $\sigma$-finite Borel measure $\mu$,  
\[
\mathbb E^{V}_xR^{V}\mu(X_{\tau_B})+R^{B}\mu(x)=R^{V}\mu(x),\quad x\in B.
\]
\end{lemma}
\begin{proof}
Without loss of generality we may and will assume that $\mu$ is bounded
(general case then follows by monotone approximation of $\mu$ by bounded measures).
By \cite[Theorem 12.16]{Sharpe}
\[
\begin{split}
\mathbb E^{V}_x\left\{\left[ R^{V}\left(\beta R^{V}_\beta\mu\right)\right](X_{\tau_B})\right\}
+\left[R^{B}\left(\beta R^{V}_\beta\mu\right)\right](x)
=\left[ R^{V}
\left(\beta R^{V}_\beta\mu\right)\right]\left(x\right),\quad x\in B.
\end{split}
\]
Hence
\[
\mathbb E^{V}_x\left\{\left[ \beta R^{V}_\beta\left( R^{V}\mu\right)\right](X_{\tau_B})\right\}+
\left[R^{B}\left(\beta R^{V}_\beta\mu\right)\right](x)
=\left[\beta R^{V}_\beta
\left( R^{V}\mu\right)\right]\left(x\right),\quad x\in B.
\]
Since $R^{V}\mu$ is $(P^{V}_t)$-excessive, we have 
\[
\beta R^{V}_\beta
\left( R^{V}\mu\right)\left(x\right)\nearrow  R^{V}\mu(x),\quad x\in V.
\]
Consequently, the most left term in the above equality 
and the term on the right-hand side of the above inequality  
converge pointwise to respective terms in the asserted equality. 
From this, we conclude that the 
second term on the left-hand side of the above equality is convergent 
pointwise in $B$  to a positive Borel function that we denote by $w$. 
Let $\eta\in C_b(B)$. By Dynkin's formula again
\[
\mathbb E^{*,V}_xR^{*,V}\eta(X_{\tau_B})+R^{*,B}\eta(x)=R^{*,V}\eta(x),\quad x\in \mathbb R^d.
\]
By \cite[Proposition 12.15]{Sharpe}, $V\ni x\mapsto \mathbb E^{*,V}_xR^{*,V}\eta(X_{\tau_B})$
is $(P^{*,V}_t)$-excessive.  Therefore, by the above equality, $R^{*,B}\eta$ is a difference 
of finite finely continuous functions, and thus, it is  finely continuous as well.
Consequently, by duality and the Lebesgue dominated convergence theorem (note that $R^{B}(\beta R^{V}_\beta\mu)\le 2R^{V}\mu$) for every bounded $\eta\in pC_b(B)$,
\[
(w,\eta)=\lim_{\beta \to \infty} (R^{B}\left(\beta R^{V}_\beta\mu\right),\eta)= 
\lim_{\beta \to \infty} (\mu,\beta R^{*,V}_\beta R^{*,B}\eta)=(\mu, R^{*,B}\eta)
=(R^{B}\mu, \eta).
\]
From this one easily concludes the result (see  \eqref{eq.ppp1}). 
\end{proof}

\section{Distributional  solutions}
\label{sec.dis}
In the paper we are concerned with distributional solutions to \eqref{eq1.1}.
This means that we would like  to consider $-Au+b\cdot\nabla u+\lambda$
as a distribution on $D\setminus K$.  
In order to do so it is not enough, in general, to impose merely local integrability of $u$ on $D\setminus K$. 
This is due to the presence of  a nonlocal term in \eqref{L} which forces some integral properties of $u$
in the whole $\mathbb R^d$.

\begin{definition}
Fix an open set $V\subset \mathbb R^d$.
For any compact $F\subset V$, we let   $r_F:=(2\text{dist}(F,V^c))\wedge 1$ and 
\[
\rho_F(x):= \nu(B^c_{r_F}\cap (F-x))+\nu(B^c_{r_F}\cap (x-F)).
\]
 We further define the linear space 
\[
\mathscr T_\nu(V):= \bigcap _{F\text{-compact},\, F\subset V} L^1_{\rho_F}(\mathbb R^d).
\]
\end{definition}

\begin{remark}
Notice that $\rho_{F_1}\le \rho_{F_2}$ whenever $F_1\subset  F_2$.  Furthermore, in case of the fractional Laplacian 
$ \rho_{F_2}\le C(F_1,F_2)\rho_{F_1}$, and as a result, 
$\mathscr T_\nu(V)=  L^1_{\rho}(\mathbb R^d)$, with $\rho(x):=\int_{B_1(x)}1\wedge|y|^2\,\nu(dy)$.

Observe that in case $\nu$ has the full support, $\mathscr T_\nu(V)\subset L^1_{loc}(\mathbb R^d)$.
In general however the inclusion does not hold (take e.g. $\nu\equiv 0$).
\end{remark}

\begin{lemma}
\label{lm.test1}
Let $A$ be a L\'evy operator and  $u\in\mathscr T_\nu(V)\cap L^1_{loc}(V)$.  Then 
\[
\int_{\mathbb R^d}|u|\,|A\eta|<\infty,\quad \eta\in C^2_c(V).
\]
\end{lemma}
\begin{proof}
Clearly, the above relation holds for operators of the form  
\eqref{L} with $\nu\equiv 0$. 
Consequently, we only need to show that for any 
$\eta\in C_c^2(V)$, $\int_{\mathbb R^d}|u|\chi_\eta<\infty$, where
\[
\chi_\eta(x):= \left|\int_{\mathbb R^d}(\eta(x)-\eta(x+y)-y\cdot\nabla\eta(x)\mathbf1_{B(0,1)})\,\nu(dy)\right|,\quad x\in\mathbb R^d.
\]
Fix  $\eta\in C_c^2(V)$ and let    $r:=\text{dist}(supp[\eta],V^c)$. 
Notice that for any $x\in V$,
\[
\left|\int_{B_{(r/2)\wedge 1}}(\eta(x)-\eta(x+y)-y\nabla\eta(x))\,\nu(dy)\right|\le \int_{B_{(r/2)\wedge1}}|y|^2\sup_{\theta\in [0,1]}|D^2\eta(x+\theta y)|\,\nu(dy)=:\chi_1(x).
\]
Clearly $\chi_1$ is compactly supported in $V$.
Now, set 
\[
\chi_2(x):= \left|\int_{B^c_{(r/2)\wedge1}}(\eta(x)-\eta(x+y))\,\nu(dy)\right|,\quad x\in\mathbb R^d,
\]
and observe that  $\chi_\eta\le \chi_1+\chi_2+\nu(B_1\setminus B_{r/2})|\nabla\eta|$.
On the other hand
\[
\chi_2(x)\le \nu(B^c_{(r/2)\wedge1})\eta(x)+\int_{B^c_{(r/2)\wedge1}}|\eta(x+y)|\,\nu(dy),\quad x\in\mathbb R^d.
\]
Consequently,
\[
\chi_\eta\le \chi_1+\nu(B_1\setminus B_{r/2})|\nabla\eta|+ \nu(B^c_{(r/2)\wedge 1})\eta(x)+\|\eta\|_\infty \nu(B^c_{(r/2)\wedge1}\cap (\text{supp}[\eta]-x)),\quad x\in\mathbb R^d.
\]
The first three terms on the right-hand side of the above inequality are compactly supported in $V$,
which  combined with fact that $u\in L^1_{loc}(V)$ implies  that their sum multiplied by $|u|$
is integrable over $\mathbb R^d$.
By the definition of the space $\mathscr T_\nu(V)$,  we have that also the most right term
above multiplied by $|u|$ is integrable in $\mathbb R^d$. 
This completes the proof.
\end{proof}

\begin{definition}
We say that $u\in \mathscr T_\nu (D)\cap L^1_{loc}(\mathbb R^d)$ is a  {\em distributional solution}
to \eqref{eq1.1} provided that (see Remark \ref{rem.adj})
\begin{equation}
\label{eq1.3}
-\int_{\mathbb R^d}u \,A^*\eta-\int_{\mathbb R^d} \text{div}b\, u\eta-
\int_{\mathbb R^d} ub\cdot \nabla \eta+\int_{\mathbb R^d}\eta\,d\lambda\ge\, 0,\quad\eta\in p\mathscr D(D\setminus K).
\end{equation}
If the above holds we write 
\[
-Au+b\nabla u+\lambda\ge 0\quad \text{in }\mathscr D'(D\setminus K).
\]
In other words $-Au+b\nabla u+\lambda$ is a positive distribution (see \cite[Definition 3.17]{DKK}).
\end{definition}

\begin{remark}
\label{rem.imp}
Let us define measure $\nu^*_x(B):=\nu(x-B)$, $x\in\mathbb R^d$ and $B\in\mathcal B(\mathbb R^d)$.
We further define positive Radon measure $\nu^{*,D}_u(dy)$ on $D$ by 
\[
\nu^{*,D}_u(dy):= \int_{D^c}u(x)\nu^*_x(dy)\,dx.
\]
Now, notice that
\[
\int_{\mathbb R^d}u \,A^*\eta=\int_{D}u \,A^*\eta+\int_D \eta\,d\nu^{*,D}_u.
\]
Consequently, by Lemma \ref{lm.test1}, $\nu^{*,D}_u$ is a positive Radon measure on $D$.
Moreover, $\nu^{*,D}_u(K)=0$ whenever $K$ is a closed polar subset of $D$
(see  \cite[Lemma 2.16]{BSW} and \eqref{eq.polar}).
\end{remark}

The following two results will be needed later. We denote $\nu_x(dy):=\nu^*_x(-dy)$.

\begin{lemma}
\label{lm.ch1}
Let $\eta, \xi\in p\mathscr B(E)$ and
$supp[\eta]\cap supp[\xi] = \emptyset$. Then
\begin{equation}
\label{eq2.weq1}
\int_{\mathbb R^d}\int_{\mathbb R^d} \eta(x)\xi(y)\nu_x(dy)\,dx=
\int_{\mathbb R^d}\int_{\mathbb R^d} \eta(y)\xi(x)\nu^*_x(dy)\,dx.
\end{equation}
\end{lemma}
\begin{proof}
The asserted  equality easily follows for smooth and compactly supported  $\eta$ and $\xi$ since then
the left hand side of \eqref{eq2.weq1} equals $\int_{\mathbb R^d} (B\xi)(x) \eta(x)\,dx$ and the right hand side 
equals $\int_{\mathbb R^d} \xi(x) (B^*\eta)(x)\,dx$, where $B$ is the L\'evy operator with triplet $(0,0,\nu)$ and
$B^*$ is its dual,  so with triplet $(0,0,\nu^*)$. The general result is a consequence of the monotone class theorem.
\end{proof}

Recall that the definition of the measure $\nu^{*,D}_u$
is given in Remark \ref{rem.imp}.

\begin{lemma}
\label{lm.hr1}
Let $V\subset\subset D$ be  an  open set,   and $u$ be a positive Borel measurable function on $\mathbb R^d$.
Then for any $\kappa\ge \kappa_0$,
\[
R^{\kappa,V}\nu^{*,D}_u(x)=\mathbb E^{\kappa}_x [\mathbf1_{D^c} u(X_{\tau_V})],\quad x\in V.
\]
\end{lemma}
\begin{proof}
The left-hand side of the asserted identity equals
\[
\int_VG^\kappa_V(x,y)\,\nu^{*,D}_u(dy)= \int_V\int_{D^c} G^\kappa_V(x,y)u(z)\,\nu^{*}_z(dy)\,dz.
\]
On the other hand, by the Ikeda-Watanabe formula (see \cite{IW}) the right-hand side of the asserted identity equals 
\[
\int_{D^c}u(y)\, P^\kappa_V(x,dy)=\int_{D^c}u(y)\int_V G^\kappa_V(x,z)\,\nu_z(dy)\,dz.
\]
Now, the result follows from Lemma \ref{lm.ch1}.
\end{proof}

\section{Local behaviour of solutions I}

Throughout the section, similarly to  Section \ref{sec4aux},  we fix $\kappa_1>\kappa_0$ and  we  drop it in the notation
(see \eqref{eq.ozn}). We shall only make  one  exception
from this rule in the assertion of Theorem \ref{th.main1} below where it would be  better to keep the full notation  
for the clarity of the presentation. 

\begin{definition}
We say that Borel measurable function $u$ on $V$ is {\em q.e. finely-continuous}  if
the process $[0,\tau_V)\ni t\mapsto u(X_t)\in \overline {\mathbb R }$ is right-continuous at $t=0$ under measure $\mathbb P_x$ for q.e. $x\in V$. 
\end{definition}

\begin{remark}
\label{rem.eqev}
Let $V$ be a Borel measurable  finely open subset of $\mathbb R^d$.
Let $u$ be a Borel measurable function that is q.e. finely continuous on $V$. 
Let $N$ be the set of those $x\in V$ for which the mapping 
$[0,\tau_V)\ni t\mapsto u(X_t)\in \overline {\mathbb R }$ is not right-continuous at $t=0$ under measure $\mathbb P_x$.
Since $N$ is polar, we have that $V\setminus N$ is finely open again. Thus, $u$ is finely continuous on $V\setminus N$.
Now applying \eqref{eq.ppp1} we obtain the following result. 
\begin{itemize}
\item If $u,v$ are    q.e. finely continuous functions on  $V$ that are equal a.e., then
$u=v$ q.e. in $V$.
\end{itemize}
We shall use this result further in the paper.
\end{remark}

\begin{definition}
We say that a $(P^V_t)$-excessive   function $h$  is $(P^V_t)$-harmonic provided that for any compact $K\subset V$, we have
\begin{equation}
\label{eq.har}
\mathbb E^V_xh(X_{\tau_K})=h(x),\quad x\in V.
\end{equation}
\end{definition}

\begin{remark}
\label{rem.harmf}
Notice that for any Borel measurable $B\subset V$,
with compact $\overline{B}$ that is a subset of $V$,  and  $(P^V_t)$-harmonic
function $h$ one has \eqref{eq.har} with $K$ replaced by $B$. Indeed, since $h$
is $(P^V_t)$-excessive, $\mathbb E^V_xh(X_{\tau_B})\le h(x),\, x\in V$ (see \cite[Theorem III.5.7]{BG}). 
On the other hand, by $(P^V_t)$-excessiveness of $h$ again  and \eqref{eq.har}
with $K=\overline{B}$, one concludes (see \cite[Theorem III.5.7]{BG} again)
\[
\mathbb E^V_xh(X_{\tau_B})\ge \mathbb E^V_xh(X_{\tau_K})=h(x),\quad x\in V. 
\]

\end{remark}

\begin{proposition}
\label{prop.harm}
Let $V$ be an open bounded subset of $\mathbb R^d$. 
Suppose that $h\in\mathscr T_\nu(V)\cap L^1(V)$ is a positive $(P^V_t)$-harmonic  function.
Then, for any $\eta\in C^2_c(V)$
\begin{equation}
\label{eq.tharm}
\int_{V}hL^*\eta=0.
\end{equation}
\end{proposition}
\begin{proof}
We let $h=0$ on $V^c$.
For any positive $\eta\in C_c^2(V)$ we have
\[
\begin{split}
-\int_VhL^{*}\eta&=-\int_VhL^{(2),*}_V\eta=\lim_{t\to 0^+}\int_Vh\frac{\eta-T^{(2),*,V}_t\eta}{t}\\&
=
\lim_{t\to 0^+}\int_Vh\frac{\eta-P^{*,V}_t\eta}{t}=
\lim_{t\to 0^+}\int_V\frac{h-P^{V}_th}{t}\eta\ge 0
\end{split}
\]
(in the second equality we used the fact that 
$\|\eta-T^{(2),*,V}_t\eta\|_\infty \le ct,\, t\in [0,1]$ for some $c\ge 0$).
By the Riesz theorem there exists a positive Radon measure $\beta$ on $V$
such that 
\begin{equation}
\label{eq1.7aa}
-\int_VhL^{*}\eta=\int_{V}\eta\,d\beta,\quad\eta\in C_c^2(V).
\end{equation}
Let $V^\delta:= \{x\in V: \text{dist}(x,\partial V)> \delta\}$.
Let $\rho_\varepsilon$ be the standard smooth mollifier. Then letting $(b \nabla h)_\varepsilon:=(b \nabla h)\ast\rho_\varepsilon$,
$h_\varepsilon:=h\ast \rho_\varepsilon$
 and $\beta_{\varepsilon}:=\beta\ast\rho_\varepsilon$, and testing $\rho_\varepsilon\ast\eta$ in \eqref{eq1.7aa},
we have for $\varepsilon\le\delta$ and $ \eta\in C_c^{\infty}(V^\delta)$,
\[
\begin{split}
\int_{V}\beta_\varepsilon \eta= \int_{V} \eta_\varepsilon\,d\beta=
(h,(-A^*+\kappa_1I)\eta_\varepsilon)+(h, (b\nabla)^* \eta_\varepsilon)= 
(h_\varepsilon,(-A^*+\kappa_1 I)\eta)+((b \nabla h)_\varepsilon, \eta).
\end{split}
\]
Set (see the comments preceding Lemma \ref{lm2.1})
\[
f^\varepsilon:= -(b\nabla h)_\varepsilon+b\nabla h_\varepsilon.
\]
Then
\[
-\int_{\mathbb R^d}Lh_\varepsilon \eta=
\int_{\mathbb R^d}(f^\varepsilon+\beta_\varepsilon)  \eta,\quad \eta\in C_c^{\infty}(V^\delta).
\]
From the above equation and continuity of $Lh_\varepsilon$, $f^\varepsilon$ and $\beta_\varepsilon$
we infer that 
\[
Lh_\varepsilon(x)=f^\varepsilon(x)+\beta_\varepsilon(x),\quad x\in V^\delta.
\]
 Then, by using Dynkin's martingale formula (see e.g. \cite[Theorem 3.9.4]{Kolokoltsov}), we conclude that  
\[
h_\varepsilon(x)= \mathbb E_xh_\varepsilon(X_{\tau_{V^\delta}})+R^{V^\delta} \beta_\varepsilon(x)-R^{V^\delta} f^\varepsilon(x),\quad x\in V^\delta.
\]
Set $\gamma_\varepsilon^\delta(x):=  \mathbb E_xh_\varepsilon(X_{\tau_{V^\delta}})$.
Since  $h$ is non-negative we have that $\gamma_\varepsilon^\delta$
is $(P^{V^\delta}_t)$-excessive (see e.g. \cite[Proposition II.2.8]{BG}). 
By \cite[Lemma 94, page 306]{DellacherieMeyer}, up to subsequence, $(\gamma_\varepsilon^\delta)$
is convergent a.e., as $\varepsilon \searrow 0$, to a $(P^{V^\delta}_t)$-excessive function $\gamma^{\delta}$. 
By Lemmas \ref{lm2.1},\ref{lm2.2}  $R^{V^\delta} \beta_\varepsilon\to R^{V^\delta}\beta$ and 
$R^{V^\delta} f^\varepsilon\to 0$ a.e. Thus, by using \eqref{eq.ppp1},
\begin{equation}
\label{eq4.4aa}
h=\gamma^{\delta}+R^{V^\delta}\beta\quad\text{ in }V^\delta.
\end{equation}
Letting $\delta\searrow 0$ we obtain, by Lemma \ref{lm2.3}, that  $R^{V^\delta}\beta\to R^{V}\beta$. 
From this and \eqref{eq4.4aa}, we conclude, by Fatou's lemma, that  $R^{V}\beta<\infty$ q.e.
By the said convergence, we have that $(\gamma^{\delta})$ is convergent too.
Let us  denote its limit by $\gamma$. We thus have
\[
h=\gamma+R^{V}\beta\quad \text{ in }V^\delta.
\]
Recall that $\gamma^{\delta}$ is $(P_t^{V^\delta})$-excessive, 
which implies (see \cite[Proposition II.2.3]{BG}) that  for any $\alpha>0$
\begin{equation}
\label{eq4.6aa}
\alpha R^{V^\delta}_\alpha (\gamma^{\delta})(x)\le \gamma^{\delta}(x),\quad x\in V^\delta.
\end{equation}
Let  $\hat  \delta\in (\delta,\infty)$. Since $V^{\hat\delta}\subset V^\delta$, we have
$\tau_{V^{\hat\delta}}\le \tau_{V^\delta}$ (see \eqref{eq.hitt}), which combined with \eqref{eq.rdef}, 
the fact that  $\gamma^{\delta}$ is non-negative, and \eqref{eq4.6aa} yields
\[
\alpha R^{V^{\hat\delta}}_\alpha (\gamma^{\delta})(x)\le \gamma^{\delta}(x),\quad x\in V^{\hat\delta}.
\]
Letting $\delta\searrow 0$ and using Lemma \ref{lm2.2}, and next letting 
$\hat\delta\searrow 0$ and using Lemma \ref{lm2.3}  give
\[
\alpha R^{V}_\alpha\gamma\le \gamma,\quad \text{a.e. in } V.
\]
Consequently, by \cite[Proposition 2.4]{BCR} there exists a $(P^V_t)$-excessive function $\hat \gamma$ such that 
$\hat \gamma= \gamma$ a.e. We deduce then, by using \eqref{eq.ppp1}, that 
\[
 h=\hat \gamma+R^V\beta,\quad \text{ in }V.
\]
By \cite[Theorem VI.2.11]{BG}
\[
h=h_0+R^V\beta_0+R^V\beta,
\]
for some positive Borel measure $\beta_0$ and positive $(P^V_t)$-harmonic function $h_0$.
By the uniqueness of the Riesz decomposition we conclude  that $\beta_0=\beta=0$.
From this and \eqref{eq1.7aa} we get \eqref{eq.tharm}.
\end{proof}

\begin{theorem}
\label{th.main1}
Suppose that either $\nu\nequiv 0$ and $u\in \mathscr T_\nu(D)\cap L^1_{loc}(\mathbb R^d)$ or
$\nu\equiv 0$ and $u\in L^1_{loc}(D\setminus K)$.
Furthermore, suppose that $u$ is  a non-negative    solution to \eqref{eq1.1}.
By the Riesz theorem there exists a positive Radon measure $ \hat \mu_0$ on $D\setminus K$
such that 
\[
-Au+b\nabla u+\lambda=\hat \mu_0\quad\text{in}\quad\mathscr D'(D\setminus K).
\]
We let  $ \mu_0(B):=\hat \mu_0(B\cap  K^c),\, B\in\mathcal B(D)$. Then
\begin{enumerate}[1)]
\item  
$u\in L^1_{loc}(D)$, $\mu_0$ is Radon on $D$ and there exists a positive bounded measure $\sigma_K$ on $D$ supported in $K$ such that
\[
-Au+b\nabla u+\lambda = \mu_0+\sigma_K\quad \text{in} \quad \mathscr D'(D);
\]
\item there exists a q.e.  finely continuous $\ell^d$-version  $\tilde u$ of $u$  that  is finite q.e. in $D$ and 
for any   open  set $V\subset\subset  D$ 
\begin{equation}
\label{eq.local12}
 \tilde u(x)=\mathbb E^{\kappa_1}_x \tilde u(X_{\tau_V})+R^{\kappa_1,V}\mu_0(x)+R^{\kappa_1,V}\sigma_K(x)-R^{\kappa_1,V}\lambda(x)+\kappa_1 R^{\kappa_1,V}u(x) \quad \text{q.e. in } V.
\end{equation}
\end{enumerate}
\end{theorem}

\begin{proof}
The proof of 1) shall be divided into four steps.
 
{\bf Step 1}. In the first step we shall show that $u$ admits a probabilistic representation  
(a form of Feynman-Kac formula). 
The reasoning in this step is similar to the one following \eqref{eq1.7aa}.
There is, however, a substantial difference between the arguments. Namely, in the proof of Proposition \ref{prop.harm}
the function $h$ was compactly supported and so its regularization  by Friedrichs mollifier was 
in the domain of the Feller generator $L$. This is, in general,  no longer true under the assumptions of the theorem.
Set $U:=D\setminus K$, and for given $\varepsilon,\delta>0$  let 
\[
D^\delta:= \{x\in D: \text{dits}(x,\partial D)>\delta\}, \quad
U^\delta:=D^\delta\setminus K, \quad U^\delta_\varepsilon:=\{x\in U^\delta: \text{dist}(x,K)>\varepsilon\}.
\]
Let  $\mu_0$ be as in the assertion of the theorem. Then 
\begin{equation}
\label{eq1.7}
\int_{\mathbb R^d}u (-A^*+\kappa_1 I)\eta-\int_{D} \text{div}b\, u\eta-
\int_{D} b u\nabla \eta+\int_{D}\eta\,d\lambda=\int_{D}\eta\,d\mu_0+\kappa_1\int_D\eta u,\quad\eta\in C_c^2(U).
\end{equation}
Set $\xi:= \mu_0-\lambda +\kappa_1 u$ (we treat $u$ as a measure $u(x)\,dx$ here).
Let $\rho_\varepsilon$ be a smooth mollifier (e.g. Friedrichs mollifier).
Testing \eqref{eq1.7} with $\eta_\varepsilon:=\eta\ast \rho_\varepsilon$ gives  for $\varepsilon\le\delta$,
\[
\int_{\mathbb R^d}u (-A^*+\kappa_1 I)\eta_\varepsilon+\int_{D}(b \nabla u)_\varepsilon \eta=\int_{D} \eta_\varepsilon\,d\xi,\quad \eta\in C_c^{\infty}(U_\varepsilon^\delta), 
\]
where  $(b \nabla u)_\varepsilon:=(b \nabla u)\ast \rho_\varepsilon$
(see the comments preceding Lemma \ref{lm2.1}). 
Set
$u^D_\varepsilon:=(\mathbf1_Du)\ast \rho_\varepsilon\in C_c^\infty(\mathbb R^d)$, and $\nu^{*,D}_{u,\varepsilon}:=\nu^{*,D}_u\ast\rho_\varepsilon$ (see Remark \ref{rem.imp}). 
Using symmetry of the convolution operator, the fact that $A^*$ is translation invariant, and  Remark \ref{rem.imp}, we conclude that
\[
\begin{split}
\int_{\mathbb R^d}u (-A^*+\kappa_1 I)\eta_\varepsilon&=\int_{\mathbb R^d}\mathbf1_Du (-A^*+\kappa_1 I)\eta_\varepsilon
+\int_{\mathbb R^d}\mathbf1_{D^c}u (-A^*+\kappa_1 I)\eta_\varepsilon\\&=
\int_{\mathbb R^d}(\mathbf1_Du)_\varepsilon (-A^*+\kappa_1 I)\eta-\int_{D}\eta_\varepsilon\,d\nu^{*,D}_u\\&=
-\int_{\mathbb R^d}(-A+\kappa_1 I)u^D_\varepsilon\eta-\int_{D}\eta\,d\nu^{*,D}_{u,\varepsilon}.
\end{split}
\]
Set  $\xi_{\varepsilon}:=\xi\ast\rho_\varepsilon$, and 
\[
f^\varepsilon:= -(b\nabla u)_\varepsilon+b\nabla u_\varepsilon.
\]
Then
\[
-\int_{\mathbb R^d}Lu^D_\varepsilon \eta=
\int_{\mathbb R^d}(f^\varepsilon+\xi_\varepsilon+\nu^{*,D}_{u,\varepsilon})  \eta,\quad \eta\in C_c^{\infty}(U_\varepsilon^\delta).
\]
Fix $\varepsilon_0>0$, and let  $\varepsilon \in (0,\varepsilon_0)$.
From the above equation and continuity of $Lu^D_\varepsilon$, $f^\varepsilon$, $\xi_\varepsilon$ and $\nu^{*,D}_{u,\varepsilon}$,
we infer that 
\[
-Lu^D_\varepsilon(x)=f^\varepsilon(x)+\xi_\varepsilon(x)+\nu^{*,D}_{u,\varepsilon}(x),\quad x\in U_{\varepsilon_0}^\delta.
\]
 Then, by using Dynkin's martingale formula (see e.g. \cite[Theorem 3.9.4]{Kolokoltsov}), we obtain that  
\[
u^D_\varepsilon(x)= \mathbb E_xu^D_\varepsilon(X_{\tau_{U_{\varepsilon_0}^\delta}})+R^{U_{\varepsilon_0}^\delta} \xi_\varepsilon(x)-R^{U_{\varepsilon_0}^\delta} f^\varepsilon(x)+R^{U_{\varepsilon_0}^\delta}\nu^{*,D}_{u,\varepsilon}(x),\quad x\in U_{\varepsilon_0}^\delta,\, \varepsilon \in (0,\varepsilon_0).
\]
Set $\gamma^\varepsilon_{\delta,\varepsilon_0}(x):=  \mathbb E_xu^D_\varepsilon(X_{\tau_{U_{\varepsilon_0}^\delta}})$.
Since  $u$ is non-negative we have that $\gamma^\varepsilon_{\delta,\varepsilon_0}$
is $(P^{U_{\varepsilon_0}^\delta}_t)$-excessive (see e.g. \cite[Proposition II.2.8]{BG}). 
By \cite[Lemma 94, page 306]{DellacherieMeyer}, up to a subsequence, $(\gamma^\varepsilon_{\delta,\varepsilon_0})$
is convergent a.e. to a $(P^{U_{\varepsilon_0}^\delta}_t)$-excessive function $\gamma_{\delta,\varepsilon_0}$. 
By Lemmas \ref{lm2.1},\ref{lm2.2}  $R^{U_{\varepsilon_0}^\delta} \xi_\varepsilon\to R^{U_{\varepsilon_0}^\delta}\xi$,
$R^{U_{\varepsilon_0}^\delta}\nu^{*,D}_{u,\varepsilon}\to R^{U_{\varepsilon_0}^\delta}\nu^{*,D}_{u}$, and 
$R^{U_{\varepsilon_0}^\delta} f^\varepsilon\to 0$ a.e. Thus,
\[
u=\gamma_{\delta,\varepsilon_0}+R^{U_{\varepsilon_0}^\delta}\xi+R^{U_{\varepsilon_0}^\delta}\nu^{*,D}_{u}\quad\text{a.e. in }U_{\varepsilon_0}^\delta.
\]

{\bf Step 2} (passing to the limit with $\varepsilon_0\to 0$).
Rearranging the terms in  the above equality, we obtain for a.e. $x\in U_{\varepsilon_0}^\delta$,
\begin{equation}
\label{eq4.4}
 u(x)+R^{U_{\varepsilon_0}^\delta}\lambda^+(x)=\gamma_{\delta,\varepsilon_0}(x)+R^{U_{\varepsilon_0}^\delta}\mu_0(x)
 +R^{U_{\varepsilon_0}^\delta}\nu^{*,D}_{u}(x)+R^{U_{\varepsilon_0}^\delta}\lambda^-(x)+\kappa_1R^{U_{\varepsilon_0}^\delta}u(x).
\end{equation}
By Lemma \ref{lm2.3} 
$R^{U_{\varepsilon_0}^\delta}\lambda^+\to R^{U^\delta}\lambda^+$,  
$R^{U_{\varepsilon_0}^\delta}\lambda^-\to R^{U^\delta}\lambda^-$,
$R^{U_{\varepsilon_0}^\delta}u\to R^{U^\delta}u$, $R^{U_{\varepsilon_0}^\delta}\nu^{*,D}_{u}\to
R^{U^\delta}\nu^{*,D}_{u}$
and $R^{U_{\varepsilon_0}^\delta}\mu_0\to R^{U^\delta}\mu_0$ as $\varepsilon_0\searrow 0$. 
By the assumptions that we made and Remark \ref{rem.imp}  $|\lambda|(U_\delta)+ \nu^{*,D}_{u}(U_\delta)<\infty$
for any $\delta>0$ which combined with    Lemma \ref{lm2.2} gives 
$R^{U^\delta}|\lambda|+ R^{U^\delta}\nu^{*,D}_{u}<\infty$ q.e. in $U_\delta$ for any $\delta>0$.
From this and \eqref{eq4.4}, we conclude, by using Fatou's lemma, that  $R^{U^\delta}\mu_0,R^{U^\delta}u<\infty$ q.e., 
and since $K$ is polar, we further infer, by using  Lemma \ref{lm2.4}, that $R^{U^\delta}\lambda=R^{D^\delta}\lambda$,
$R^{U^\delta}u=R^{D^\delta}u$,
$R^{U^\delta}\nu^{*,D}_{u}=R^{D^\delta}\nu^{*,D}_{u}$,
and $R^{U^\delta}\mu_0=R^{D^\delta}\mu_0$  q.e. in $D^\delta$. 
By the  convergences stated after \eqref{eq4.4}, we get that $(\gamma_{\delta,\varepsilon_0})$ is convergent 
as $\varepsilon_0\to 0$, too.
We denote its limit by $\gamma_{\delta}$. We thus have
\begin{equation}
\label{eq4.5}
 u(x)=\gamma_{\delta}(x)+R^{D^\delta}\xi(x)+R^{D^\delta}\nu^{*,D}_{u}(x)\quad \text{for a.e. }x\in D^\delta,
\end{equation}
with all the terms finite a.e. in $D^\delta$.

{\bf Step 3} ($\gamma_\delta$ has an excessive version in $D^\delta$).
Recall that $\gamma_{\delta,\varepsilon_0}$ is $(P_t^{U_{\varepsilon_0}^\delta})$-excessive, 
which implies (see \cite[Proposition II.2.3]{BG}) that  for any $\alpha>0$
\begin{equation}
\label{eq4.6}
\alpha R^{U_{\varepsilon_0}^\delta}_\alpha (\gamma_{\delta,\varepsilon_0})(x)\le \gamma_{\delta,\varepsilon_0}(x),\quad x\in U_{\varepsilon_0}^\delta.
\end{equation}
Fix $\hat\varepsilon_0\in (\varepsilon_0,\delta)$. Since $U_{\hat \varepsilon_0}^\delta\subset U_{\varepsilon_0}^\delta$, we have
$\tau_{U_{\hat \varepsilon_0}^\delta}\le \tau_{U_{\varepsilon_0}^\delta}$ (see \eqref{eq.hitt}), which combined with \eqref{eq.rdef}, the fact that  $\gamma_{\delta,\varepsilon_0}$
is non-negative, and \eqref{eq4.6} yields
\[
\alpha R^{U_{\hat \varepsilon_0}^\delta}_\alpha (\gamma_{\delta,\varepsilon_0})(x)\le \gamma_{\delta,\varepsilon_0}(x),\quad x\in U_{\varepsilon_0}^\delta.
\]
Letting $\varepsilon_0\searrow 0$ and using Lemma \ref{lm2.2}, and then letting 
$\hat \varepsilon_0\searrow 0$ and using Lemma \ref{lm2.3}  give
\[
\alpha R^{U^\delta}_\alpha\gamma_\delta\le \gamma_\delta,\quad \text{a.e. in } U^\delta.
\]
This in turn, by Lemma \ref{lm2.4}, implies that 
\[
\alpha R^{D^\delta}_\alpha\gamma_\delta\le \gamma_\delta,\quad \text{a.e. in } D^\delta.
\]
Consequently, by \cite[Proposition 2.4]{BCR} there exists a $(P^{D^\delta}_t)$-excessive function $\hat \gamma_\delta(\cdot)$ such that 
$\hat \gamma_\delta(\cdot)= \gamma_\delta(\cdot)$ a.e. Thus, by \eqref{eq4.5}, 
\begin{equation}
\label{eq.now1}
 u=\hat \gamma_\delta+R^{D^\delta}\xi+R^{D^\delta}\nu^{*,D}_{u},\quad \text{a.e. in }D^\delta.
\end{equation}
From this and \cite[Proposition VI.2.3]{BG} we conclude  in particular that $u\in L^1_{loc}(D)$
(this is the part of the assertion in case $\nu\equiv 0$).

{\bf Step 4} (conclusion of the result). By \cite[Theorem VI.2.11]{BG}
\[
\hat\gamma_\delta=h_\delta+R^{D^\delta}\sigma_\delta\quad\text{in }D^\delta
\]
for some positive Borel measure $\sigma_\delta$ and positive $(P^{D^\delta}_t)$-harmonic function $h_\delta$.
Consequently
\begin{equation}
\label{eq.delta}
 u=h_\delta+R^{D^\delta}\sigma_\delta+R^{D^\delta}\xi+R^{D^\delta}\nu^{*,D}_{u},\quad \text{a.e. in }D^\delta.
\end{equation}
Fix  $\eta\in C_c^2(D^\delta)$. Multiplying the above equality by (the bounded continuous function) $L^*\eta$ and applying Proposition \ref{prop.harm} and Lemma \ref{lm.duoper} yields
\[
-\int_{D^\delta}uL^*\eta=\int_{D^\delta}h_\delta L^*\eta+\int_{D^\delta} R^{D^\delta}\sigma_\delta L^*\eta+
\int_{D^\delta} R^{D^\delta}\xi L^*\eta=\int_{D^\delta}\eta\,d\sigma_\delta+\int_{D^\delta}\eta\,d\xi+\int_{D^\delta}\eta\,d\nu^{*,D}_{u}.
\]
Hence
\[
-\int_{D^\delta}uL^*\eta=\int_{D^\delta}\eta\,d\sigma_\delta+\int_{D^\delta}\eta\,d\xi+\int_{D^\delta}\eta\,d\nu^{*,D}_{u},\quad \eta\in C_c^2(D^\delta).
\]
As a result, by Helly's selection theorem (see e.g. \cite[Theorem 5.19]{Kallenberg}),
there exists a positive Radon measure $\sigma$ on $D$ such that, up to a subsequence, $\sigma_\delta\to \sigma$
in the vague topology (i.e. $\int_{D}\eta\,d\sigma_{\delta}\to \int_D\eta\,d\sigma$ for any $\eta\in C_c(D)$), which in turn 
combined with the previous equality yields
\[
-\int_{D} uL^*\eta=\int_{D}\eta\,d\sigma+\int_{D}\eta\,d\xi+\int_{D}\eta\,d\nu^{*,D}_{u}\quad \eta\in C_c^2(D).
\]
Hence and by Remark \ref{rem.imp}
\[
-\int_{\mathbb R^d} uL^*\eta=\int_{D}\eta\,d\sigma+\int_{D}\eta\,d\xi,\quad \eta\in C_c^2(D).
\]
Invoking   equality \eqref{eq1.7}
we conclude  that   $supp[ \sigma]\subset K$. This  completes the proof of 1) with 
$\sigma_K:=\sigma$.

As to 2) let us fix open  $V\subset\subset D$. 
We keep the notation of the proof of 1). 
Let $\delta>0$ be such that $\overline V\subset D^{\delta}$.
Set $N_1:=\{x\in V: R^{D^\delta}|\xi|(x)+R^{D^\delta}\nu^{*,D}_{u}(x)=\infty\}$. By Step 2 of the proof of 1) the set $N_1$ is polar.  
Furthermore, by \cite[Proposition II.3.5]{BG}, $N_2:=\{x\in V: \hat \gamma_\delta(x)=\infty\}$ is polar too.
Letting $\tilde u(x):=\hat \gamma_\delta(x)+R^{D^\delta}\xi(x)+R^{D^\delta}\nu^{*,D}_{u}(x)$ for $x\in V\setminus(N_1\cup N_2)$ and zero otherwise
we get, by \eqref{eq.now1},  that $\tilde u$ is a q.e. finely continuous $\ell^d$-version of $u$ that is finite q.e. in $V$
($R^{D^\delta}\xi^+$, $R^{D^\delta}\xi^-$, $R^{D^\delta}\nu^{*,D}_{u}$, $\hat \gamma_\delta$ are excessive, hence finely continuous). 
This being so for any  open  set $V\subset\subset D$, we easily find a q.e. finely continuous $\ell^d$-version $\tilde u$
of $u$ on $D$ that is q.e. finite in $D$. Thus, by \eqref{eq.delta} and Remark \ref{rem.eqev}, we have 
\begin{equation}
\label{eq.deltat}
 \tilde u=h_\delta+R^{D^\delta}\sigma_\delta+R^{D^\delta}\xi+R^{D^\delta}\nu^{*,D}_{u},\quad \text{q.e. in }D^\delta.
\end{equation}
From this, Dynkin's formula (see \cite[Theorem 12.16]{Sharpe}) and the definition of a harmonic function
(see Remark \ref{rem.harmf}), we conclude that 
\[
\begin{split}
\mathbb E^{D^\delta}_x[\tilde u(X_{\tau_V})]& =\mathbb E_x^{D^\delta}[h_\delta(X_{\tau_V})]+
\mathbb E_x^{D^\delta}[R^{D^\delta}\sigma_\delta(X_{\tau_V})]\\&\quad +\mathbb E_x^{D^\delta}[R^{D^\delta}\xi(X_{\tau_V})] 
+\mathbb E_x^{D^\delta}[R^{D^\delta}\nu^{*,D}_{u}(X_{\tau_V})] 
\\&
=h_\delta(x)+R^{D^\delta}\sigma_\delta(x)-R^{V}\sigma_\delta(x)\\&\quad+R^{D^\delta}\xi(x)-R^{V}\xi(x)+R^{D^\delta}\nu^{*,D}_{u}(x)-R^{V}\nu^{*,D}_{u}(x)
,\quad \text{q.e. in }V.
\end{split}
\]
This equality together  with \eqref{eq.deltat} yield
\begin{equation}
\label{eq.deltat2}
 \tilde u(x)=\mathbb E^{D^\delta}_x[\tilde u(X_{\tau_V})] +R^{V}\sigma_\delta(x)+R^{V}\xi(x)+R^{V}\nu^{*,D}_{u}(x)\quad \text{q.e. in } V.
\end{equation}
By the definition of the measures $\mathbb P^{D^\delta}_x, \mathbb P^{D}_x$ we have
\[
\mathbb E^{D^\delta}_x[\tilde u(X_{\tau_V})] =\mathbb E_x[\tilde u(X_{\tau_V})\mathbf1_{\{\tau_V<\tau_{D^\delta}\}}] 
\to \mathbb E_x[\tilde u(X_{\tau_V})\mathbf1_{\{\tau_V<\tau_{D}\}}]= \mathbb E^{D}_x[\tilde u(X_{\tau_V})].
\]
From this, \eqref{eq.deltat2} and Lemma \ref{lm2.2} we infer that
\[
 \tilde u(x)=\mathbb E^{D}_x[\tilde u(X_{\tau_V})] +R^{V}\sigma(x)+R^{V}\xi(x)+R^{V}\nu^{*,D}_{u}(x)\quad \text{q.e. in } V.
\]
Now, using Lemma \ref{lm.hr1} gives the desired equality. 
\end{proof}

\section{Local behavior of solutions II, maximum principle}

In the following section  we shall prove that the formula appearing  in assertion  2) 
of Theorem \ref{th.main1} holds true for $\kappa_1=0$.
The advantage of this refinement  is that the term $\kappa_1 u$     
that is present on the right-hand side of the said formula vanishes.
The second advantage is a maximum principle that we can obtain,  as a simple corollary,
with a constant that in some instances is
independent of the drift $b$.  The said  refinement requires an additional assumption formulated as
\eqref{eq.exgreen} below that is in fact a necessary condition for having all the integrals in "$\kappa_1=0$ version"
of \eqref{eq.local12} finite.

\subsection{Auxiliary results}

Let $V$ be an  open   subset of $\mathbb R^d$.
As we mentioned in the comment following Hypothesis \ref{eq.stanass}, each $R^{0,V}_\alpha(x,dy)$
is absolutely continuous with respect to the Lebesgue measure on $V$, which in turn implies that 
$R^{0,V}(x,dy):= \lim_{\alpha\to 0^+} R^{0,V}_\alpha(x,dy)$ (monotone limit) shares this  property too.  
By \cite[Theorem V.58, page 52]{DellacherieMeyer1}  for  any $\kappa\ge 0$ 
there exists a positive function $r_{\kappa}(\cdot,\cdot)\in \mathscr B(V)\otimes \mathscr B(V)$ 
such that 
\begin{equation}
\label{eq.rr1}
R^{0,V}_\kappa(x,dy)=r_{\kappa}(x,y)\,dy,\quad x\in V.
\end{equation}
As  mentioned in Section \ref{sec.fprs}
for any $\kappa\ge \kappa_0,\, \alpha\ge 0$  there exists Green's function $G^{\kappa}_{V,\alpha}$
of $A-b\nabla+(\kappa+\alpha) I$ on $V$ (see also \eqref{eq.greenred}). 
We used at that point the result from \cite{BG} that required   the dual operator
$(A-b\nabla)^*+\kappa I$  to generate a Markov semigroup.  
This in turn holds true provided that $\kappa\ge \kappa_0$.  
We shall extend the family $(G^{\kappa}_{V,\alpha}(\cdot,\cdot))$
to $\kappa$ ranging over the set $[0,\infty)$. In order to do this, we define (see \eqref{eq.gvvve} and \eqref{eq.greenred})
\begin{equation}
\label{eq.gfcas}
G^0_{V,\alpha}(x,y):=
\begin{cases}
 G_{V,\alpha-\kappa_0}^{\kappa_0}(x,y),\quad \alpha\ge \kappa_0\\
 G^{\kappa_0}_{V,0}(x,y)+(\kappa_0-\alpha)\int_V r_{\alpha}(x,z)G^{\kappa_0}_{V,0}(z,y)\,dz,\quad \alpha\in[0,\kappa_0).
\end{cases}
\end{equation}

\begin{proposition}
\label{prop.4c}
The family of functions defined by \eqref{eq.gfcas}  satisfies:
\begin{enumerate}
\item[(i)] for any $x\in V,\, \alpha\ge 0$, $G^{0}_{V,\alpha}(x,\cdot)$ is excessive with respect to $(P^{\kappa_0\vee \alpha,*,V}_t)$;
\item[(ii)] for any $y\in V,\, \alpha\ge 0$, $G^{0}_{V,\alpha}(\cdot,y)$ is $\alpha$-excessive with respect to $(P^{0,V}_t)$;
\item[(iii)] for any $x\in V,\, \alpha\ge 0$, $R^{0,V}_\alpha(x,dy)=G^{0}_{V,\alpha}(x,y)\,dy$;
\item[(iv)] for any $x,y\in V$,  $\alpha,\beta\ge 0$, 
\begin{equation}
\label{eq.4G}
\begin{split}
G^{0}_{V,\alpha}(x,y)&=G^{0}_{V,\beta}(x,y)+(\beta-\alpha)\int_VG^{0}_{V,\alpha}(x,z)G^{0}_{V,\beta}(z,y)\,dz
\\&=G^{0}_{V,\beta}(x,y)+(\beta-\alpha)\int_VG^{0}_{V,\beta}(x,z)G^{0}_{V,\alpha}(z,y)\,dz.
\end{split}
\end{equation}
\end{enumerate}
\end{proposition}
\begin{proof}
From  the resolvent identity and \eqref{eq.rr1}, we conclude that for any $\alpha,\beta\ge 0$  and any  $x\in V$,
\begin{equation}
\label{eq.GG}
\begin{split}
r_\alpha(x,y)&= r_\beta(x,y)+(\beta-\alpha)\int_V r_\alpha(x,z)r_\beta(z,y)\,dz\\&
= r_\beta(x,y)+(\beta-\alpha)\int_V r_\beta(x,z)r_\alpha(z,y)\,dz,\quad y\in V,\, \text{a.e.}
\end{split}
\end{equation} 
When $\alpha\ge \kappa_0$, we clearly have
$r_{\alpha}(x,\cdot)=G^{\kappa_0}_{V,\alpha-\kappa_0}(x,\cdot)$ a.e.  (see \eqref{eq.rr1} and \eqref{eq.red2}).
This combined with \eqref{eq.rr1}, \eqref{eq.gfcas} and \eqref{eq.GG} gives (iii) that may be restated
as follows
\begin{equation}
\label{eq.rrr1}
r_{\alpha}(x,\cdot)=G^{0}_{V,\alpha}(x,\cdot) \quad \text{a.e. in }V\text{ for any }x\in V,\, \alpha\ge 0.
\end{equation} 
As to (i), note  that each $G^{\kappa_0}_{V,\alpha}(x,\cdot)$ is  excessive 
with respect to $(P^{\alpha+\kappa_0,*,V}_t)$,  by the very definition of the Green function, see \eqref{eq.GGG}.
Observe also that by \eqref{eq.gfcas} for any $\alpha\in [0,\kappa_0)$, $ G_{V,\alpha}^0(x,\cdot)$ is 
excessive with respect to  $(P^{\kappa_0,*,V}_t)$ (due to excessiveness   of  
$G^{\kappa_0}_{V,0}(x,\cdot)$ with respect to $(P^{\kappa_0,*,V}_t)$). This establishes (i).

(iv) By (i), we have in particular that  $G_{\alpha,V}^0(x,\cdot)$ is co-finely continuous for any $x\in V$
and $\alpha\ge 0$. From this, \eqref{eq.rrr1}, \eqref{eq.GG} and \eqref{eq.ppp1} (for co-fine topology) one easily concludes (iv).

(ii)  For any  $y\in V$ and  $\alpha\ge \kappa_0$,  $G^0_{V,\alpha}(\cdot,y)$
is $\alpha$-excessive with respect to $(P^{0,V}_t)$  
(again by the very definition of the Green function $G^{\kappa_0}_{V,\alpha}$).
What is left is to show that (ii) holds true for any $\alpha\in [0,\kappa_0)$.
By \eqref{eq.4G}
\[
\begin{split}
\beta R^{0,V}_{\alpha+\beta}& (G^0_{V,\alpha}(\cdot,y))(x)+G^0_{V,\beta+\alpha}(x,y)\\&\quad=\beta\int_V G^0_{V,\beta+\alpha}(x,z)G^0_{V,\alpha}(z,y)\,dz+G^0_{V,\beta+\alpha}(x,y)
=G^0_{V,\alpha}(x,y),\quad x,y\in V.
\end{split}
\]
As a result, $\beta R^{0,V}_{\alpha+\beta} (G^0_{V,\alpha}(\cdot,y))(x)\le  G^0_{V,\alpha}(x,y)$
for any $\alpha,\beta\ge 0$, $x,y\in V$. By the second equality in \eqref{eq.4G}, with $\beta=\kappa_0$,
 $G^0_{V,\alpha}(\cdot,y)$ is $\kappa_0$-excessive with respect to $(P^{0,V}_t)$, hence finely continuous on $V$
 for any $\alpha\in [0,\kappa_0)$.
 Consequently,  by \cite[Exercise 4.23]{Sharpe}, for any $\alpha\in [0,\kappa_0)$, $G^0_{V,\alpha}(\cdot,y)$ is $\alpha$-excessive with respect to $(P^{0,V}_t)$.
\end{proof}

Thanks to the above result we may define  function  
\begin{equation}
\label{eq.extmes}
R^{0,V}_\alpha\mu(x):= \int_VG^{0}_{V,\alpha}(x,y)\,\mu(dy),\quad x\in V
\end{equation}
for any $\alpha\ge 0$  and positive Borel measure $\mu$ on $\mathbb R^d$ 
(clearly $R^{0,V}_\alpha\mu=R^{\alpha,V}\mu$ for $\alpha\ge \kappa_0$,  see \eqref{dual.res2}).

\begin{lemma}
\label{lm.id1}
Let $V\subset\mathbb R^d$ be an open  set, and $u$ be a positive Borel measurable function on $\mathbb R^d$.
Then for any $0\le \alpha<\kappa$,
\[
(\kappa-\alpha)(R^{0,V}_\alpha(\mathbb E^\kappa_\cdot u(X_{\tau_V})))(x)+\mathbb E^\kappa_x u(X_{\tau_V})=\mathbb E^{0}_xe^{-\alpha\tau_V}u(X_{\tau_V}),\quad x\in V.
\]
\end{lemma}
\begin{proof}
By making use of  the Lebesgue monotone convergence theorem we may and will assume 
without loss of generality that $u$ is bounded. By \eqref{eq2.killex} and strong Markov property one has
\[
\begin{split}
R^{0,V}_\alpha(\mathbb E^\kappa_\cdot u(X_{\tau_V}))(x)&=R^{0,V}_\alpha(\mathbb E^{0}_\cdot [e^{-\kappa \tau_V}u(X_{\tau_V})])(x)\\&=
\mathbb E^0_x\int_0^{\tau_V}e^{-\alpha s}\mathbb E^0_{X_s}\left[e^{-\kappa\tau_V}u(X_{\tau_V})\right]\,ds
\\&=\mathbb E^0_x\int_0^{\tau_V}e^{-\alpha s}\left[e^{-\kappa(\tau_V-s)}u(X_{\tau_V})\right]\,ds\\&=
\mathbb E^{0}_x\left[e^{-\kappa\tau_V}u(X_{\tau_V}) \int_0^{\tau_V}e^{(\kappa-\alpha) s}\,ds\right]
\\&=\frac{1}{\kappa-\alpha}\mathbb E^{0}_x\left[e^{-\kappa\tau_V}u(X_{\tau_V})(e^{(\kappa-\alpha)\tau_V}-1)\right].
\end{split}
\]
From this and \eqref{eq2.killex} again,  we infer the asserted equality.
\end{proof}

Now, thanks to \eqref{eq.extmes}, we may extend  Lemmas \ref{lm2.3}, \ref{lm2.4}
to the case $\kappa_1=0$.

\begin{lemma}
\label{lm2.3zero}
Let  $(V_n)$ be an increasing sequence of  open subsets of $\mathbb R^d$ and let 
$V:=\bigcup_{n\ge 1} V_n $. Let $\beta$ be a positive Borel measure on $V$. Then
\[
R^{0,V_n}\beta(x)\nearrow R^{0,V}\beta (x),\quad x\in V.
\]
\end{lemma}
\begin{proof}
Let $\kappa\ge \kappa_0$. By the resolvent identity,
\[
R^{0,V_n}\beta(x)=R^{\kappa,V_n}\beta(x)+\kappa R^{0,V_n}R^{\kappa,V_n}\beta(x),\quad x\in V_n.
\]
By Lemma \ref{lm2.3} $f_n(x):=R^{\kappa,V_n}\beta(x)\nearrow R^{\kappa,V}\beta(x)=:f(x),\, x\in V$.
Hence, there exists $g\in p\mathscr B(V)$ such that 
\[
R^{0,V_n}f_n\nearrow g,\quad \text{in }V.
\]
Clearly, $g\le R^{0,V}f$ in $V$. On the other hand, for $n\ge m$
\[
R^{0,V_n}f_n\ge R^{0,V_n}f_m\nearrow R^{0,V}f_m.
\]
Letting $m\to \infty$ yields $g\ge R^{0,V}f$. Thus, $g= R^{0,V}f$. As a result
\[
R^{0,V_n}\beta(x)\nearrow R^{\kappa,V}\beta(x)+\kappa R^{0,V}R^{\kappa,V}\beta(x)=R^{0,V}\beta(x),\quad x\in V.
\]
This completes the proof.
\end{proof}

Clearly, in order    to have the integrals in \eqref{eq.local12} well defined with $\kappa_1=0$,
we need $G_{V,0}^0$ to be finite a.e., or in other words
that
\[
\lim_{\alpha\to 0^+} G_{\alpha,V}^0(x,y)=\sup_{\alpha>0} G_{\alpha,V}^0(x,y)<\infty \text{ for a.e. }(x,y)\in V\times V
\]
 The last property is equivalent to (see \cite[Theorem 2, Section 3.7]{Chung})
 \begin{equation}
 \label{eq.exgreen}
 \mathbb P^0_x(\tau_V<\infty)=1,\quad x\in V.
 \end{equation}
 The above condition holds whenever $\nu(B^c(0,r))>0$ for any $r\ge 1$ (see  \cite[Proposition 3.7]{SW1}),
 thus it holds true for operators from Examples \ref{ex11}--\ref{ex13}.
 For other conditions based on the  Fourier symbol see e.g.  \cite[Theorem 5.5]{BSW}.
 From the  result mentioned  we may deduce that \eqref{eq.exgreen} holds for the operators in Example \ref{ex14}.
By \eqref{eq.exgreen} and \cite[Theorem 4, Section 3.6]{Chung} for any finite
$(P^{0,V}_t)$-excessive function $f$, we have $P^{0,V}_t f(x)\to 0,\, t\to \infty$ for $x\in V$
such that $f(x)<\infty$,
which implies that 
 \begin{equation}
 \label{eq.exgreent}
\lim_{\alpha\to 0}\alpha R^{0,V}_\alpha f(x)=0,\quad x\in V,\, f(x)<\infty.
 \end{equation}
 We will use this property in the proof of the  main theorem of the section.
 
\begin{theorem}
\label{th.main2}
Assume that  for any open  $V\subset\subset D$  \eqref{eq.exgreen} holds  and $R^{0,V}\lambda^+\in L^1(V)$.
Suppose that $u\in \mathscr T_\nu(D)\cap L^1_{loc}(\mathbb R^d)$ (or $u\in L^1_{loc}(D\setminus K)$ in case $\nu\equiv 0$)
is a non-negative  distributional solution to \eqref{eq1.1}.  
Let $ \mu_0$ be as in Theorem \ref{th.main1} and 
\[
U:=\bigcup_{V\subset\subset D} \{x\in V: R^{0,V}\lambda^+(x)<\infty\},\quad N:=D\setminus U.
\]
Then $N$ is polar,  $u$ has an $\ell^d$-version $\hat u$ that is finely continuous on 
$D\setminus N$ and  letting $\hat u(x):= 0,\, x\in N$, we have that for any open $V\subset\subset D$,
\begin{equation}
\label{eq.main3}
\hat u(x)+R^{0,V}\lambda^+(x)=\mathbb E^{0}_x \hat u(X_{\tau_V})+R^{0,V}\mu_0(x)+R^{0,V}\sigma_K(x)+R^{0,V}\lambda^-(x) \quad x\in V,
\end{equation}
and 
\begin{equation}
\label{eq.main4}
\hat u(x)+R^{0,V}\lambda^+(x)\ge \inf_{y\in D\setminus V} \hat u(y) \mathbb P^0_x (\tau_V<\tau_D),\quad x\in V.
\end{equation}
If for any open $V\subset\subset D$ there exists $\kappa\ge 0 $ such that 
$R^{0,V}_{\kappa}\lambda^+(x)<\infty,\, x\in V$, 
then $u$ has an $\ell^d$-version $\hat u$ that is finely continuous in $D$.
\end{theorem}
\begin{proof}
Observe that \eqref{eq.main4} follows directly from \eqref{eq.main3}. 
Indeed,
\[
\hat u(x)+R^{0,V}\lambda^+(x)\ge \mathbb E^{0}_x \hat u(X_{\tau_V})\ge  \mathbb E^{0}_x \mathbf1_{D\setminus V}(X_{\tau_V})\hat u(X_{\tau_V})
\ge  \inf_{y\in D\setminus V} \hat u(y) \mathbb P^0_x (\tau_V<\tau_D),   \quad x\in V.
\]
The rest of the proof shall  be divided into four steps. 

{\bf Step 1}($N$ is polar). First observe that since $\mathbb R^d$, as a separable metric space, is second countable,
we may chose a countable   base $(V_n)_{n\ge 1}$   of open subsets of $\mathbb R^d$, and for any such base we have
\[
N=D\setminus \bigcup_{V_n\subset\subset D} \{x\in V_n: R^{0,V_n}\lambda^+(x)<\infty\}.
\]
By the assumptions made each $R^{0,V_n}\lambda^+$ is finite a.e. in $V_n$.
Consequently, by  \cite[Proposition VI.2.3]{BG} $R^{0,V_n}\lambda^+<\infty$ q.e. in $V_n$ for any $n\ge 1$.
This in turn combined with the above equality gives that $N$ is polar.

{\bf Step 2}(fine continuity). 
By Theorem \ref{th.main1} there exists  $\tilde u$ that is  a q.e. finely continuous $\ell^d$-version of 
$u$ such that   for any open $V\subset\subset D$ and $\kappa>\kappa_0$,
\[
 \tilde u(x)+R^{0,V}_\kappa\lambda^+(x) =\mathbb E^{\kappa}_x \tilde u(X_{\tau_V})+R^{0,V}_\kappa\mu_0(x)+R^{0,V}_\kappa\sigma_K(x)+\kappa R^{0,V}_\kappa \tilde u(x)+R^{0,V}_\kappa\lambda^-(x)
\quad \text{q.e. in } V.
\]
The second term on the left-hand side and each term on the right-hand side of the  
above equality is a $(P^{\kappa,V}_t)$-excessive function, hence finely continuous on $V$.
We denote all the terms on the right-hand side of the  
above equality by $w_i,\, i=1,\dots, 5$ starting from the most left term. 
By the definition of $N$, $R^{0,V}_{\kappa}\lambda^+$ is finite on $D\setminus N$ for any open $V\subset\subset D$.
Thus, letting $\hat u:=\sum_{i=1}^5w_i-R^{0,V}_\kappa\lambda^+$ on $V\setminus N$ and zero on $N$,
we have that $\tilde u=\hat u$ q.e. and 
\begin{equation}
\label{eq.point1}
\hat  u(x)=\mathbb E^\kappa_x \hat  u(X_{\tau_V})+R^{0,V}_\kappa\mu_0(x)+R^{0,V}_\kappa\sigma_K(x)+\kappa R^{0,V}_\kappa \hat u(x)-R^{0,V}_\kappa\lambda(x) \quad  x\in V\setminus N.
\end{equation}
Denote for a moment  so constructed  $\hat u$  by $\hat u_V$. Now let $V_1\subset \subset D$
be an   open set and  $V\subset V_1$.  Let $\hat u_{V_1}$ be the function constructed in the same way as $\hat u_V$
but with $V$ replaced by $V_1$. Clearly, $\hat u_{V_1}$ is finely continuous on $V\setminus N$. 
Therefore, by \eqref{eq.ppp1}, $\hat u_{V}=\hat u_{V_1}$ on $V\setminus N$.
Consequently, there exists a finely continuous function $\hat u$ on $D$
such that \eqref{eq.point1} holds for any  open $V\subset\subset D$.

In case for any open $V\subset\subset D$ there exists $\kappa\ge 0 $ such that 
$R^{0,V}_{\kappa}\lambda^+(x)<\infty,\, x\in V$ we may obviously take $N=\emptyset$ in the construction of $\hat u$
($\kappa$ may vary when changing $V$).

{\bf Step 3} (main formula: the case $\lambda^+\equiv 0$) 

We additionally assume that $\lambda^+\equiv 0$. In this case $N=\emptyset$.
Let $V$ be an  open subset of $D$ such that $V\subset\subset D$.
By the assumption that  we made and \eqref{eq.point1} $\hat u(x)\ge \kappa R^{0,V}_\kappa \hat u(x),\, x\in V$
for any $\kappa>\kappa_0$. As a result, by the resolvent identity and fine continuity of $\hat u$, we obtain that 
$\hat u$ is $(P^{0,V}_t)$-excessive. Let $\hat u(x)<\infty$ and $\alpha\in (0,\kappa]$.
Notice that, by the resolvent identity,  for any positive Borel measure $\beta$
if $R^{0,V}_\alpha R^{0,V}_\kappa\beta(x)+R^{0,V}_\kappa\beta(x)<\infty$, then $R^{0,V}_\alpha\beta(x)<\infty$.
Consequently,  by \eqref{eq.point1} and since $\alpha R^{0,V}_\alpha\hat u(x)\le \hat u(x)<\infty$,
one concludes that $R^{0,V}_\alpha\beta(x)<\infty$ with $\alpha\in (0,\kappa]$ and $\beta=\mu_0,\sigma_K,\lambda^-$.
Keeping this in mind, we put  
 $R^{0,V}_\alpha$ to both sides of \eqref{eq.point1} and apply  the resolvent identity to get 
\[
\begin{split}
R^{0,V}_\alpha\hat  u(x)&=R^{0,V}_\alpha(\mathbb E^\kappa_\cdot \hat  u(X_{\tau_V}))(x)\\&\quad
+\frac{1}{\kappa-\alpha}(R^{0,V}_\alpha\mu_0(x)-R^{0,V}_\kappa\mu_0(x))
+\frac{1}{\kappa-\alpha}(R^{0,V}_\alpha\sigma_K(x)-R^{0,V}_\kappa\sigma_K(x))\\&\quad+\frac{\kappa}{\kappa-\alpha}(R^{0,V}_\alpha \hat u(x)-R^{0,V}_\kappa \hat u(x))+\frac{1}{\kappa-\alpha}(R^{0,V}_\alpha\lambda^-(x)-R^{0,V}_\kappa\lambda^-(x))\\&
=[R^{0,V}_\alpha(\mathbb E^\kappa_\cdot \hat  u(X_{\tau_V}))(x)+\frac{\kappa}{\kappa-\alpha}R^{0,V}_\alpha \hat u(x)+\frac{1}{\kappa -\alpha}\mathbb E_x\hat u(X_{\tau_V})]\\&\quad
-\frac{1}{\kappa-\alpha}[R^{0,V}_\kappa\mu_0(x)+R^V\sigma_K(x)-R^{0,V}_\kappa\lambda^-(x)+\mathbb E^\kappa_\cdot \hat  u(X_{\tau_V}))(x)-\kappa R^{0,V}_\kappa \hat u(x))]
\\&\quad+\frac{1}{\kappa-\alpha}(R^{0,V}_\alpha\mu_0(x)+R^{0,V}_\alpha\sigma_K(x)+R^{0,V}_\alpha\lambda^-(x))\\&
=[R^{0,V}_\alpha(\mathbb E^\kappa_\cdot \hat  u(X_{\tau_V}))(x)+\frac{\kappa}{\kappa-\alpha}R^{0,V}_\alpha \hat u(x)+\frac{1}{\kappa -\alpha}\mathbb E^\kappa_x\hat u(X_{\tau_V})]-\frac{1}{\kappa-\alpha}\hat u(x)
\\&\quad+\frac{1}{\kappa-\alpha}(R^{0,V}_\alpha\mu_0(x)+R^{0,V}_\alpha\sigma_K(x)+R^{0,V}_\alpha\lambda^-(x)).
\end{split}
\]
In the last equation we used \eqref{eq.point1}. Consequently, multiplying both sides of the above equality by
$(\kappa-\alpha)$, rearranging the terms and next applying   Lemma \ref{lm.id1} yield
\[
\begin{split}
\hat  u(x)-\alpha R^{0,V}_\alpha \hat u(x) &=[(\kappa-\alpha) R^{0,V}_\alpha(\mathbb E_\cdot \hat  u(X_{\tau_V}))(x)+\mathbb E_x\hat u(X_{\tau_V})]\\&\quad
+R^{0,V}_\alpha\mu_0(x)+R^{0,V}_\alpha\sigma_K(x)+R^{0,V}_\alpha\lambda^-(x)\\&=
\mathbb E^0_x \hat  u(X_{\tau_V})
+R^{0,V}_\alpha\mu_0(x)+R^{0,V}_\alpha\sigma_K(x)+R^{0,V}_\alpha\lambda^-(x),\quad x\in V.
\end{split}
\]
Letting $\alpha\to 0^+$ and using \eqref{eq.exgreent} give \eqref{eq.main3} for any $x\in V$ such that $\hat u(x)<\infty$.
Since both  the left and the right-hand side of \eqref{eq.main3} are $(P^{0,V}_t)$-excessive, hence finely continuous,
one obtains by \eqref{eq.ppp1} the desired equality.

{\bf Step 4} (the general case).
Let $D_\varepsilon \subset\subset D$ be an open set such that $K\subset D_\varepsilon$.
Set $w_\varepsilon(x):= R^{0,D_\varepsilon}\lambda^+(x),\, x\in D_\varepsilon$.
By the assumptions of the theorem $w_\varepsilon$ is finite on $D_\varepsilon\setminus N$ and $w_\varepsilon\in L^1(D_\varepsilon)$.
Let $\eta\in C_c^\infty(D_\varepsilon)$. Applying   the resolvent identity one concludes 
\[
(w_\varepsilon,-L_\kappa^*\eta)=(R^{0,D_\varepsilon}_\kappa\lambda^++\kappa R^{0,D_\varepsilon}_\kappa w_\varepsilon ,-L^*_\kappa\eta)=
(\lambda^+,\eta)+\kappa(w_\varepsilon,\eta).
\]
Consequently, 
\[
(w_\varepsilon,-L_0^*\eta)=(\lambda^+,\eta),\quad \eta\in C_c^\infty(D_\varepsilon).
\]
As a result, $v_\varepsilon:=u+w_\varepsilon$ is a non-negative distributional solution to
\[
-L_0v-\lambda^-\ge 0\quad\text{in }\, D_\varepsilon \setminus K.
\]
Therefore, by Step 3, for any $V\subset\subset D_\varepsilon$,
\[
\hat v_\varepsilon(x)=\mathbb E^{0}_x \hat v_\varepsilon (X_{\tau_V})+R^{0,V}\mu_0(x)+R^{0,V}\sigma_K(x)+R^{0,V}\lambda^-(x) \quad x\in V,
\]
(note that $\sigma_K$  in Theorem \ref{th.main1} does not depend on $D$ as long as $K\subset D$).
Since, by Step 3,  $\hat v_\varepsilon$ is $(P^{0,D_\varepsilon}_t)$-excessive, we have
\[
\hat v_\varepsilon(x)= \lim_{\alpha\to \infty} \alpha R^{0,D_\varepsilon}_\alpha(\hat v_\varepsilon)(x)=
 \lim_{\alpha\to \infty} \alpha R^{0,D_\varepsilon}_\alpha(u+w_\varepsilon)(x)
 = \lim_{\alpha\to \infty} \alpha R^{0,D_\varepsilon}_\alpha(u)(x)+w_\varepsilon(x)
\]
for any $x\in D_\varepsilon\setminus N_\varepsilon$, where  $N_\varepsilon:=\{x\in D_\varepsilon: w_\varepsilon(x)=\infty\}$.
We let $\hat u_\varepsilon (x):= \lim_{\alpha\to \infty} \alpha R^{0,D_\varepsilon}_\alpha(u)(x)$
for any $x\in D_\varepsilon\setminus N_\varepsilon$ and zero on $N_\varepsilon$ (note that $N_\varepsilon\subset N$).
Clearly, $\hat u_\varepsilon$ is finely continuous on $D_\varepsilon\setminus N_\varepsilon$ as
\[
\hat u_\varepsilon(x)=\hat v_\varepsilon(x)-w_\varepsilon(x),\quad x\in D_\varepsilon\setminus N_\varepsilon,
\]
both $\hat v_\varepsilon, w_\varepsilon$ are $(P^{0,D_\varepsilon}_t)$-excessive, hence finely continuous in $D_\varepsilon$,
and $w_\varepsilon$ is finite on $D_\varepsilon\setminus N_\varepsilon$.
Now, by Lemma \ref{lm.dynmeas}
\[
\hat u_\varepsilon(x)+R^{0,V}\lambda^+(x)=\mathbb E^{0}_x \hat u_\varepsilon  (X_{\tau_V})+R^{0,V}\mu_0(x)+R^{0,V}\sigma_K(x)+R^{0,V}\lambda^-(x) \quad x\in V\setminus N_\varepsilon.
\]
Since $\hat u_\varepsilon $ and $\hat u$ constructed in Step 2 are finely continuous $\ell^d$-versions of $u$ in $D_\varepsilon\setminus N$,
we conclude by \eqref{eq.ppp1} that $\hat u_\varepsilon=\hat u$ on $D_\varepsilon\setminus N$. 
By arbitrariness  of $D_\varepsilon$, we get \eqref{eq.main3} on $V\setminus N$ for any open $V\subset\subset D$.
Of course, \eqref{eq.main3} holds on $N$ as well since on this set the left hand side gives infinity. 
\end{proof}

\begin{remark}
Observe that, letting $\varphi\equiv 1$, we have
\[
w_V(x):=\mathbb P^0_x(\tau_V<\tau_D)=\mathbb E^{0,D}_x\varphi(X_{\tau_V}),\quad x\in D.
\]
By \cite[Proposition 12.15]{Sharpe} $w_V$ is a $(P^{0,D}_t)$-excessive function.
Thus, whenever the {\em strong maximum principle} holds  for $(P^{\kappa,D}_t)$:
\begin{equation}
\label{smp}
f\in p\mathscr B_b(D), \, R^{\kappa,D}f(x)=0 \text{ for some } x\in D\quad\Rightarrow\quad R^{\kappa,D}f\equiv 0\text{ in }D,
\end{equation}
then there exists $c>0$ such that $w_V(x)\ge c,\, x\in V$. Indeed, suppose that $w_V(x_0)=0$ for some $x_0\in D$.
Since $w_D$ is  $(P^{0,D}_t)$-excessive , we have $\alpha R^{\alpha,D}w_V\nearrow w_V$ in $D$.
Since $w_V(x_0)=0$ we have $\alpha R^{\alpha,D}w_V\nearrow w_V(x_0)=0$. As a result, by the strong maximum principle,  $\alpha R^{\alpha,D}w_V\equiv 0$ in $D$, which implies that $w_V\equiv 0$. 
On the other hand $w_V(x)=1,\, x\in D\setminus \overline V$, a contradiction. Therefore $w_V$ is 
strictly positive in $D$. By the strong Feller property $w_V$ is lower semicontinuous in $D$,
which in turn gives that $\inf_{x\in V}w_V(x)>0$.
\end{remark}

\begin{remark}
Observe that if for any $x\in D$, $G_D(x,y)>0$ for a.e. $y\in D$, then \eqref{smp} holds. 
Therefore, by  \cite[Theorem 2.5]{KK}, $\inf_{x\in V}w_V(x)>0$ for the operator of Example \ref{ex11}. 
\end{remark}

\begin{remark}
Observe that 
\[
w_V(x)=\mathbb P^0_x(\tau_V<\tau_D)=\mathbb P^0_x(X_{\tau_V}\in D\setminus V)\ge \mathbb P^0_x(X_{\tau_V}\in D\setminus \overline{V})= \int_VG_V(x,y)\nu_y(D\setminus \overline{V})\,dy.
\]
If we knew that $\nu_y(D\setminus \overline{V})>0$ a.e. in $y\in V$, then we would have that $w_V(x)>0,\, x\in V$.
We see that $w_V(x)=1,\, x\in (D\setminus V)^r$ (so called regular points for $D\setminus V$; see \cite[Definition I.11.1]{BG}). By \cite[Proposition II.3.3]{BG}, $(D\setminus V)\setminus(D\setminus V)^r$
is semipolar, and by the very definition of regular points,  $D\setminus \overline{V}\subset (D\setminus V)^r\subset (D\setminus V)$
(see the comments following  \cite[Definition I.11.1]{BG}). We have
\[
D=V\cup (D\setminus V)^r\cup[(D\setminus V)\setminus(D\setminus V)^r].
\]
Since $w_V$ is $(P^{0,D}_t)$-excessive
\[
\mathbb E^{0}_x\int_0^{\tau_D}w_V(X_s)\,ds=R^{0,D}_1w_V(x)\le w_V(x),\quad x\in D.
\]
By \cite[Proposition II.3.4]{BG}, the set $\{t\ge 0: X_t\in (D\setminus V)\setminus(D\setminus V)^r\}$ is a.s. countable.
Consequently, $R^{0,D}_1w_V(x)>0,\, x\in D$, which implies that $w_V(x)>0,\, x\in D$. Since $w_V$
is lower semicontinuous, we conclude that $\inf_{x\in V}w_V(x)>0$. In particular, the last relation holds true for the operator of
Example \ref{ex12}.
\end{remark}

\begin{corollary}
Let $\hat u$  and $V$ be as in the assertion of Theorem \ref{th.main2}. 
\begin{enumerate}[1)]
\item If $\nu\equiv 0$ (i.e. $A$ is a local operator), then 
\[
\hat u(x)\ge \inf_{y\in D\setminus V} \hat u(y),\quad x\in V
\]
\item In general, we have 
\[
\hat u(x)\ge \inf_{y\in D\setminus V} \hat u(y)\Big[1- \int_{V}G_V(x,y)\nu_y(D^c)\,dy\Big]
\ge \inf_{y\in D\setminus V} \hat u(y)\Big[1- \sup_{y\in V}\nu_y(D^c) \mathbb E_x\tau_V\Big] ,\quad x\in V.
\]
\end{enumerate}
\end{corollary}

\subsection*{Acknowledgements}
{\small This work was supported by Polish National Science Centre
(Grant No.  2022/45/B/ST1/01095).}


\begin{thebibliography}{32}




\bibitem{ADGW}
Ao, W.,  DelaTorre, A., González, M.,  Wei, J.:
A gluing approach for the fractional Yamabe problem with isolated singularities.
{\em J. Reine Angew. Math.} {\bf 763} (2020) 25--78.

\bibitem{Aviles}
Aviles, P.: A study of the singularities of solutions of a class of nonlinear elliptic partial
differential equations. {\em  Comm. Partial Differential Equations} {\bf 7} (1982) 609--643.

\bibitem{BP}
Baras, P., Pierre P.: Singularites eliminables pour des equations semi-lineaires.
{\em Ann. Inst. Fourier} {\bf 34} (1984) 182--206.

\bibitem{Bers}
Bers. L., Isolated singularities of minimal surfaces. {\em  Ann. of Math.} {\bf 53} (1951) 364--386.


\bibitem{Bertoin}
Bertoin, J.: {\em  L\'evy processes.} Cambridge University Press (1996).

\bibitem{BCR}
Beznea, L., C\^impean, I., and R\"ockner, M.:
Irreducible recurrence, ergodicity, and extremality of invariant measures for resolvents.
{\em  Stoch. Proc. and their Appl.} {\bf 128} (2018) 1405--1437


\bibitem{bil}
Billingsley, P.:  {\em Convergence of probability measures. }
Second edition. Wiley Series in Probability and Statistics: Probability and Statistics. 
Wiley-Interscience Publication. John Wiley \& Sons, Inc., New York (1999).

\bibitem{BG}
Blumenthal, R. M.,   Getoor, R. K.: {\em Markov processes and potential
theory}, Academic Press, New York and London (1968).

\bibitem{BG1}
Blumenthal, R. M.,   Getoor, R. K.: Sample functions of stochastic processes with stationary independent increments.
{\em  J. Math. Mech.} {\bf  10} (1961)  493--516.

\bibitem{BSK}
Bogdan, K. Sztonyk, P.; Knopova, V.: 
Heat kernel of anisotropic nonlocal operators.
{\em Doc. Math.} {\bf 25} (2020) 1–54.

\bibitem{Bocher}
B\^ocher, M.:  Singular points of functions which satisfy partial differential equations of the elliptic type.
{\em Bull. Amer. Math. Soc.} {\bf 9} (1903) 455--465

\bibitem{BSW}
B\"ottcher, B., Schilling, R., Wang, J.: L\'evy Matters III. L\'evy-Type Processes: Construction,
Approximation and Sample Path Properties. {\em Lecture Notes in Math.} {\bf 2099}
Springer, Cham (2013).


\bibitem{BL}
Brezis, H.,  Lions, P.-L.:  A note on isolated singularities for linear elliptic equations.
{\em Adv. Math. Suppl. Stud.}  (1981).

\bibitem{BN}
Brezis, H.  Nirenberg, L., Removable singularities for nonlinear elliptic equations.
{\em Topol. Methods Nonlinear Anal.}  {\bf 9} (1997) 201--219.

\bibitem{BV}
Brezis, H.,  Veron, L.:
Removable singularities for some nonlinear elliptic equations.
{\em Arch. Rational Mech. Anal.} {\bf 75} 1--6 (1980).








\bibitem{CGS}
Caffarelli, L. A.,    Gidas, B.,  Spruck, J.:  Asymptotic symmetry and local behavior
of semilinear elliptic equations with critical Sobolev growth. {\em  Comm. Pure Appl. Math.}{\bf  42}
(1989) 271--297.


\bibitem{CJS}
Caffarelli, L.,  Jin, T.,  Sire, Y.,  Xiong, J.:
Local analysis of solutions of fractional semi-linear elliptic equations with isolated singularities.
{\em Arch. Ration. Mech. Anal.} {\bf 213} (2014) 245–268.

\bibitem{CLN}
Caffarelli, L., ; Li, Y., Nirenberg, L.: 
Some remarks on singular solutions of nonlinear elliptic equations III: viscosity solutions including parabolic operators.
{\em Comm. Pure Appl. Math.} {\bf  66} (2013) 109--143.



\bibitem{CH}
Chen, H., V\'eron, L.:Singularities of fractional Emden's equations via Caffarelli-Silvestre extension.
{\em J. Differential Equations} {\bf  361} (2023) 472--530.








\bibitem{CL}
Chen, C. C.,  Lin, C.S.: 
Local behavior of singular positive solutions of semilinear elliptic equations with Sobolev exponent.
{\em Duke Math. J.} {\bf 78} (1995) 315--334.

\bibitem{CQ}
Chen, H.,  Quaas, A.: 
Classification of isolated singularities of nonnegative solutions 
to fractional semi-linear elliptic equations and the existence results.
{\em J. Lond. Math. Soc.} {\bf  97} (2018) 196–221.


\bibitem{Chung}
Chung, K. L.,  Walsh, J.: 
{\em Markov processes, Brownian motion, and time symmetry.}
Second edition, Grundlehren Math. Wiss.  {\bf 249} (Fundamental Principles of Mathematical Sciences)
Springer, New York  (2005) 

\bibitem{DS}
DeGiorgi, E.  Stampacchia, G.:  Sulle singolarita eliminabili delle ipersuperficie
minimali.{\em  Atti Accad. Naz. Lincei Rend. Cl. Sci. Fis. Mat. Natur.} {\bf 38} (1965) 352--357.


\bibitem{DDGW}
DelaTorre, A., del Pino, M.,  González, M.,  Wei, J.:
Delaunay-type singular solutions for the fractional Yamabe problem.
{\em Math. Ann.} {\bf 369} (2017) 597--626.

\bibitem{DG}
DelaTorre, A., Gonz\'alez, M.:
Isolated singularities for a semilinear equation for the fractional Laplacian arising in conformal geometry.
{\em Rev. Mat. Iberoam.} {\bf 34} (2018) 1645--1678.

\bibitem{DellacherieMeyer}
Dellacherie, C.,  Meyer, P.A.:  {\em Probabilities and Potential C}.
North-Holland, Amsterdam (1988).

\bibitem{DellacherieMeyer1}
Dellacherie, C.,  Meyer, P.A.:  {\em Probabilities and Potential B}.
North-Holland, Amsterdam (1982).

\bibitem{DL}
DiPerna, R. J.,  Lions, P.-L.:
Ordinary differential equations, transport theory and Sobolev spaces.
{\em Invent. Math.} {\bf 98} (1989)  511--547.

\bibitem{DKK}
J.J. Duistermaat,
J.A.C. Kolk: {\em Distributions Theory and Applications.}
Springer New York Dordrecht Heidelberg London (2010).


\bibitem{Getoor}
Getoor, R.K.:  Measure perturbations of Markovian semigroups, {\em Potential Anal.} {\bf 11} (1999) 101--133.
 
 \bibitem{GiSp}
Gidas, B.,   Spruck, J.: Global and local behavior of positive solutions of nonlinear elliptic
equations.{\em  Comm. Pure Appl. Math.} {\bf  34} (1981) 525–598.


\bibitem{GS}
 Gilbarg, D.,  Serrin, J.: On isolated singularities of solutions of second order elliptic differential equations
{\em J. Analyse Math}. {\bf 4} (1955/56).



\bibitem{GMS}
Gonz\'alez, M., Mazzeo, R., Sire, Y.: Singular solutions of fractional order conformal Laplacians. 
{\em J. Geom. Anal.} {\bf 22} 845--863 (2012).

\bibitem{HN}
Hansen, W.,  Netuka, I.: Semipolar Sets and Intrinsic Hausdorff Measure
{\em Potential Analysis} {\bf 51} (2019) 49--69.

\bibitem{HW}
Hartman, P.,  Wintner, A.: 
On the local behavior of solutions of non-parabolic partial differential equations I,II,III, 
{\em Amer. J. Math.} {\bf 75}, {\bf 76}, {\bf 77} (1953),(1954),(1955).



\bibitem{IW}
Ikeda, N,  Watanabe, S.: On some relations between the harmonic measure and the L\'evy measure for a certain class of Markov processes. {\em  J. Math. Kyoto Univ.} {\bf  2} (1962) 79-95.

\bibitem{Ishikawa}
Ishikawa, Y.: Density estimate in small time for jump processes with singular L\'evy measures and its support property.
{\em Funct. Differ. Equ.}  {\bf 8} (2001)  273–285.

\bibitem{JDSX}
Jin, T.,  de Queiroz, O. S., Sire, Y., Xiong, J.:
On local behavior of singular positive solutions to nonlocal elliptic equations.
{\em Calc. Var. Partial Differential Equations} {\bf 56} (2017).

\bibitem{Kallenberg}
Kallenberg, O.: {\em Foundations of modern probability.}  Second edition. Probability and its Applications (New York) Springer-Verlag, New York (2002).

\bibitem{Kurtz}
Kurtz, T.G.  Equivalence of Stochastic Equations and Martingale Problems. 
In: Crisan, D. (eds) Stochastic Analysis 2010. Springer, Berlin, Heidelberg (2011).

\bibitem{kuhn} 
K\"uhn, F.:  On martingale problems and Feller processes.
 {\em Electron. J. Probab.} {\bf 23} (2018)  1-18.



\bibitem{KK}
Knopova, V.,  Kulik, A.:
Parametrix construction of the transition probability density of
the solution to an SDE driven by $\alpha$-stable noise.
{\em Annales de l’Institut Henri Poincar\'e - Probabilit\'es et Statistiques}  {\bf  54}
(2018) 100--140.

\bibitem{KKS}
Knopova, V.,  Kulik, A., Schilling, R. L.:  Construction and heat kernel estimates of general stable-like Markov processes.
{\em Dissertationes Math.} {\bf  569} (2021).

\bibitem{Kolokoltsov}
Kolokoltsov, V.N.: {\em Markov Processes, Semigroups and Generators. }
 Walter de Gruyter GmbH \& Co. KG, Berlin/New York (2011).
 
 \bibitem{KPP}
Kulik, A., Peszat, S.,  Priola, E.: 
Gradient formula for transition semigroup corresponding to stochastic equation 
driven by a system of independent L\'evy processes. 
{\em NoDEA Nonlinear Differential Equations Appl.} {\bf 30} (2023).
 

\bibitem{Congming}
Li, C.:  Local asymptotic symmetry of singular solutions to nonlinear elliptic equations.
{\em Invent. Math.} {\bf 123} (1996) 221--231.

\bibitem{Labutin}
Labutin, D. A.:
Removable singularities for fully nonlinear elliptic equations.
{\em Arch. Ration. Mech. Anal.} {\bf 155} (2000) 201--214.

\bibitem{LWX}
Li, C. Wu, Z.,  Xu, H.:
Maximum principles and B\^ocher type theorems.
{\em Proc. Natl. Acad. Sci. USA} {\bf  115} (2018) 6976--6979.



  
\bibitem{LLWX}
Li, C.,  Liu, C., Wu, Z., Xu, H.:
Non-negative solutions to fractional Laplace equations with isolated singularity.
{\em Adv. Math. } {\bf 373} (2020).


\bibitem{LW}
Li, J.,  Wan, F.: B\^ocher-type theorem on $n$-dimensional manifolds with conical metric.
{\em Proc. Amer. Math. Soc.} {\bf 147} (2019) 4527--4538.




\bibitem{LB}
Li, Y., Bao, J.:
Local behavior of solutions to fractional Hardy-H\'enon equations with isolated singularity
{\em Ann. Mat. Pura Appl.} {\bf  198} (2019) 41--59.

\bibitem{Lions}
Lions, P.-L.: Isolated singularities in semilinear problems
{\em J. Differential Equations} {\bf 38} (1980).

\bibitem{Pazy}
Pazy, A.: {\em Semigroups of Linear Operators and Applications to Partial Differential
Equations.} Springer-Verlag, New York (1983).

\bibitem{Picard}
Picard, J.: Density in small time at accessible points for jump processes.
{\em Stochastic Process. Appl.} {\bf  67} (1997) 251--279.

\bibitem{SW1}
Schilling, R.L.,   Wang, J.: Strong Feller continuity of Feller processes and semigroups. 
{\em Infin. Dimens. Anal. Quantum Probab. Relat. Top.}
{\bf 15}  (2012).

\bibitem{SW}
Serrin, J.,  Weinberger, H. F.:
Isolated singularities of solutions of linear elliptic equations.
{\em Amer. J. Math.} {\bf 88} (1966).

\bibitem{Serrin1}
Serrin, J.: Local behavior of solutions of quasi-linear equations.
{\em Acta Math.} {\bf 111} (1964) 247--302.


\bibitem{Serrin2}
Serrin, J.: Isolated singularities of solutions of quasi-linear equations.
{\em Acta Math.} {\bf 113} (1965) 219--240.


\bibitem{Sharpe}
Sharpe, M.: {\em General Theory of Markov Processes}, Academic Press, New
York (1988).


\bibitem{Situ}
Situ, R.: {\em Theory of Stochastic Differential Equations with Jumps
and Applications.} Mathematical and Analytical Techniques with
Applications to Engineering, Springer, New York, NY, USA (2005).


\bibitem{SZ}
Song, Y.,  Zhang, X.:
Regularity of density for SDEs driven by degenerate L\'evy noises.
{\em Electron. J. Probab.} {\bf  20}  (2015). 




\bibitem{Taylor}
Taylor, M.: 
B\^ocher's theorem with rough coefficients.
{\em Potential Anal.} {\bf  56} (2022) 65--86.

\bibitem{VV}
V\'azquez, J. L., V\'eron, L.:
Isolated singularities of some semilinear elliptic equations.
{\em J. Differential Equations} {\bf  60} (1985) 301--321.


\bibitem{Veron}
V\'eron, L.:
Solutions singuli\`eres d'\'equations elliptiques semilin\'eaires.
{\em C. R. Acad. Sci. Paris S\'er. A-B} {\bf 288} (1979)  A867--A869.


\bibitem{Wan}
Wan, F.: 
B\^ocher-type theorems for the Poisson's equation on manifolds with conical metrics.
{\em Calc. Var. Partial Differential Equations} {\bf  59} (2020).


\bibitem{WXZ}
Wang, F.-Y., Xu, L.,  Zhang, X.:
Gradient estimates for SDEs driven by multiplicative L\'evy noise.
{\em J. Funct. Anal.} {\bf  269} (2015) 3195--3219.


\bibitem{Veron1}
Veron, L., Singularities of solutions of second order quasilinear equations.
Longman, Harlow, 1996.


\bibitem{Xiong}
Xiong, J.:
The critical semilinear elliptic equation with isolated boundary singularities.
{\em J. Differential Equations} {\bf 263} (2017) 1907--1930.


\bibitem{YZ}
Yang, H., Zou, W.:
On isolated singularities of fractional semi-linear elliptic equations.
{\em Ann. Inst. H. Poincar\'e C Anal. Non Lin\'eaire} {\bf  38} (2021) 403--420.

















































































\end{thebibliography}
\end{document}